\documentclass[12pt]{article}
\usepackage{amsmath, amsthm, amssymb}
\usepackage{bbm}
\usepackage[mathscr]{eucal}
\usepackage{rotating}

\usepackage{tikz}
\usepackage{tikz-cd}
  \usetikzlibrary{matrix,arrows}

\newcommand\Cite[2] {\cite[#1]{#2}}

\def\apo           {\mbox{\sc s}}

\def\be            {\begin{equation}}
\def\bearl         {\begin{array}{l}}
\def\bearll        {\begin{array}{ll}}
\def\Bl            {{\ensuremath{\mathrm{Bl}}}}
\def\BlF           {\Bl^{\scriptscriptstyle(F)}}
\def\boti          {\,{\boxtimes}\,}
\def\C             {{\ensuremath{\mathcal C}}}
\def\cc            {^{\scriptscriptstyle\mathrm C}}
\def\ccc           {^{\scriptscriptstyle\,\mathrm C}\!}

\def\cir           {\,{\circ}\,}
\newcommand\comcat[1] {(#1\,{\downarrow}\,\UC)}
\def\complex       {{\ensuremath{\mathbbm C}}}

\def\Corr          {\mathrm v^{}_{\!F}}
\def\CWm           {\ensuremath{\mathscr F\!\mathscr M}}
\def\D             {{\ensuremath{\mathcal D}}}

\def\dsty          {\displaystyle }
\def\edist         {e_\uparrow}
\def\ee            {\end{equation}}

\def\eear          {\end{array}}

\def\EndD          {{\ensuremath{\mathrm{End}_\D}}}
\def\eps           {\varepsilon}
\def\epF           {\varepsilon_F}
\def\eq            {\,{=}\,}
\newcommand\eqpic[4]{$$ \raisebox{-4pt}{\begin{picture}(#2,#3) #4 \end{picture}} $$}
\newcommand\erf[1] {(\ref{#1})}
\def\findim        {fi\-ni\-te-di\-men\-si\-o\-nal}
\def\fine          {fine}

\def\Fun           {{\mathcal F\hspace*{-2pt}un}}

\newcommand\gpq[3] {\surf^{#1}_{#2{\sss|}#3}}
\def\Hom           {{\ensuremath{\mathrm{Hom}}}}
\def\HomD          {{\ensuremath{\mathrm{Hom}_\D}}}
\newcommand\hsp[1] {\mbox{\hspace{#1 em}}}
\def\id            {\mbox{\sl id}}

\def\idFv          {\ensuremath{\id_{F^\vee_{}}}}
\def\idsm          {\mbox{\footnotesize\sl id}}

\def\idXv          {\ensuremath{\id_{X^\vee_{}}}}

\def\ii            {{\rm i}}
\def\iK            {\imath^K}

\def\iKFe          {\imath^K_{F^\Wee_{}}}

\def\iN            {\,{\in}\,}
\newcommand\includepic[1] {{\begin{picture}(0,0)(0,0) \scalebox{.304}
                   {\includegraphics{pic_fuSc19_#1.eps}}\end{picture}}}

\newcommand\ite[1] {\item[{\rm({}W#1{\rm})}]}
\newcommand\itm[1] {\item[{\rm({}M#1}{\rm)}]}
\def\jef           {\jmath^\eps_F}
\def\ko            {{\ensuremath{\Bbbk}}}
\newcommand\labl[1]{\label{#1}\ee}
\def\mA            {{\rm A}}
\def\Map           {{\mathrm{Map}}}
\def\maps          {\ensuremath{\Map(\surf)}}
\def\mapsC         {\ensuremath{\Map\cc(\surf)}}
\def\mapsx         {\ensuremath{\mathrm{MAP}(\surf)}}
\def\mark          {\varGamma}
\def\marke         {\mark_{\!\surf}}
\def\markep        {\mark_{\!\surf'}}
\def\markstd       {\mark^{\circ}}
\def\mB            {{\rm B}}
\def\mBl           {\widetilde{\Bl}}
\def\mBlF          {\widetilde{\Bl}{}^{\scriptscriptstyle(F)}}
\def\mC            {{\rm C}}
\def\mCorr         {\widetilde{\mathrm v}^{}_{\!F}}
\def\mF            {{\rm F}}
\def\mi            {\,{-}\,}
\def\mOne          {\widetilde\Delta_\ko}
\def\mS            {{\rm S}}
\def\ms            {\mso_{\surf}}
\def\mso           {{\mathrm s}}
\newcommand\mss[2] {\mso_{\surf_{#1},\surf_{#2}}}
\def\mSurf         {\ensuremath{\slm\mathcal S\slurf}}
\def\mSurfC        {\ensuremath{\slm\mathcal S\slurf\ccc}}
\def\mT            {{\rm T}}
\def\mvs           {\widetilde{\vs}}
\def\mZ            {{\rm Z}}
\def\naturals      {{\ensuremath{\mathbb N}}}

\newcommand\nxl[1] {\\[#1mm]}
\newcommand\Nxl[1] {\\[-1.3em]\\[#1mm]}
\newcommand\oline[2]{\mathrm{c\!ut}_{#2}(#1)}
\def\omF           {\omega_F}
\def\one           {{\bf1}}
\def\One           {{\Delta_\ko}}

\def\op            {^{\mathrm{op}}}
\def\oti           {\,{\otimes}\,}

\def\otik          {\,{\otimes_\ko}\,}

\def\PhF           {\varPhi_{\!F}}
\def\Phinv         {\Phi_{\!F}^{-1}}

\def\piode         {\pi_0(\partial\surf_1)}
\def\piods         {\pi_0(\partial\surf)}
\def\piodsi        {\pi_0(\surfi)}
\def\piodso        {\pi_0(\surfo)}
\def\piv           {\pi}   
\def\pl            {\,{+}\,}
\def\qquand        {\qquad{\rm and}\qquad}

\def\slm           {\mbox{\sl m}}
\def\slurf         {\hspace*{-1pt}\mbox{\sl urf}}
\def\slz           {\ensuremath{\mathrm{SL}(2,\zet)}}
\def\Sqcup         {\mbox{\large$\cup$}}
\def\sse           {\scriptsize }
\def\sS            {S^\circ}
\def\sSne          {S^\circ_{n;\eps}}
\def\sss           {\scriptscriptstyle }
\def\surf          {{E}}
\def\Surf          {\ensuremath{\mathcal S\slurf}}
\def\SurfC         {\ensuremath{\mathcal S\slurf\ccc}}
\def\surfi         {{\partial_{\rm in}E}}
\def\surfo         {{\partial_{\rm out}E}}

\def\surfcm        {\ensuremath{(\surf,C,\mark)}}
\def\surfm         {{\ensuremath{(\surf,\mark)}}}
\def\surfmp        {\ensuremath{(\surf,\mark')}}
\def\surfn         {{\ensuremath{\surf,\mark}}}
\def\surfpm        {\ensuremath{(\surf',\mark')}}

\def\Times         {\,{\times}\,}
\def\To            {\,{\to}\,}

\def\U             {U}
\def\UC            {\U\cc} 
\def\Vect          {\ensuremath{\mathcal V\hspace*{-1pt}\mbox{\sl ect}}}
\def\vs            {{\mathrm v}} 
\def\vsms          {\Corr(\surf)}
\def\vsmsp         {\Corr(\surf')}
\def\wee           {^{\Wee}}
\def\Wee           {\eps}

\def\zet           {{\ensuremath{\mathbb Z}}}

\newtheorem{thm}{Theorem}
\newtheorem{lem}[thm]{Lemma}

\newtheorem{prop}[thm]{Proposition}

\newtheorem*{thmx}{Proposition \ref{prop-0} and Theorem \ref{thm-g}}
\theoremstyle{definition}
\newtheorem{defi}[thm]{Definition}
\newtheorem{rem}[thm]{Remark}

\begin{document}

\numberwithin{equation}{section}
\numberwithin{thm}{section}


\thispagestyle{empty} 
\begin{flushright}
   {\sf ZMP-HH/16-5}\\
   {\sf Hamburger$\;$Beitr\"age$\;$zur$\;$Mathematik$\;$Nr.$\;$588}\\[2mm]
\end{flushright}
\vskip 3.0em
\begin{center}\Large
{\bf CONSISTENT SYSTEMS OF CORRELATORS IN \\[3pt]
   NON-SEMISIMPLE CONFORMAL FIELD THEORY}
\end{center}\vskip 2.2em

\begin{center}
  J\"urgen Fuchs $^{a}$
  ~~and~~ Christoph Schweigert $^b$
\end{center}
\vskip 9mm

\begin{center}
  $^a$ Teoretisk fysik, \ Karlstads Universitet
  \\Universitetsgatan 21, \ S\,--\,\,651\,88 Karlstad
 \\[7pt]
  $^b$ Fachbereich Mathematik, \ Universit\"at Hamburg\\
  Bereich Algebra und Zahlentheorie\\
  Bundesstra\ss e 55, \ D\,--\,20\,146\, Hamburg
\end{center}
\vskip 4.4em
\noindent{\sc Abstract}
\\[4pt]
Based on the modular functor associated with a -- not necessarily semisimple
-- finite non-de\-ge\-nerate ribbon category $\mathcal D$, we present a definition
of a consistent system of bulk field correlators for a conformal field theory which
comprises invariance under mapping class group actions and compatibility with the sewing
of surfaces.  We show that when restricting to surfaces of genus zero such systems 
are in bijection with commutative symmetric Frobenius algebras in $\mathcal D$,
while for surfaces of any genus they are in bijection with modular Frobenius algebras
in $\mathcal D$. This provides additional insight into structures familiar from rational 
conformal field theories and extends them to rigid logarithmic conformal field theories.

\newpage

\section{Introduction and main result}

A crucial task in any quantum field theory is to establish the existence of a consistent
system of correlators of the fields of the theory. In two-dimensional conformal field 
theories these correlators are specific elements in suitable spaces of conformal blocks,
characterized by the fact that they satisfy various consistency conditions.
The spaces of conformal blocks of a conformal field theory can be explored from
several different mathematical points of view. The approach relevant to the present paper
describes them as \findim\ vector spaces that carry projective representations
of mapping class groups of surfaces with marked points and are compatible with the
sewing of surfaces. These vector spaces are constructed in terms of morphism spaces
of a braided monoidal category \D\ \cite{lyub11,BAki}. We refer to the data coming 
with the spaces of conformal blocks as the \emph{monodromy data} based on \D\
(they are also known as chiral data, or as Moore-Seiberg data).

Specifically, we consider local conformal field theories on closed oriented surfaces. For
these, the fields are called bulk fields, and a consistent choice of bulk fields provides an
object of \D, which we denote by $F$ and to which for brevity we refer as the \emph{bulk object}.
In this paper we give a precise mathematical realization of the notions of bulk object and 
of systems of correlators of bulk fields for a conformal field theory corresponding to a given 
category \D. A novelty of our approach is that \D\ does not have to be semisimple. 

That the conformal block spaces are \findim\ is a non-trivial and useful finiteness 
property. Let us mention that this property is satisfied for the categories \D\ that are
relevant to interesting classes of conformal field theories, including in particular all
rational conformal field theories, but also a large class of models for which \D\ is
non-semisimple. Because of the analytic properties of their conformal blocks, the latter 
models go under the name of \emph{logarithmic} conformal field theories (see e.g.\ 
\cite{gura6}). The fact that our approach does not require \D\ to be semisimple thus 
makes it relevant for important applications of logarithmic
conformal field theories like e.g.\ the study of critical dense polymers \cite{dupl2}.

\medskip

In this paper we develop a precise definition of the notion of consistency of a system
of bulk field correlators (see Definition \ref{def:Corr}). Then we prove

\begin{thmx}~\nxl2
{\rm (i)}\, Let \D\ be a finite ribbon category and $F$ an object of \D.
The consistent systems of gen\-us-ze\-ro bulk field correlators for monodromy data 
based on \D\ and with bulk object $F$ are in bijection with structures of a
commutative symmetric Frobenius algebra on $F$.
\nxl2
{\rm (ii)}\, Let \D\ be a modular finite ribbon category and $F$ an object of \D.
The consistent systems of bulk field correlators for monodromy data based on \D\ and
with bulk object $F$ are in bijection with structures of a modular Frobenius algebra
{\rm(}in the sense of Definition \rm{\ref{def:modFrob})} on $F$.
\end{thmx}

Consistency conditions for correlators have been discussed extensively in the
conformal field theory literature (see e.g.\ \cite{frsh2,sono2,lewe3}). They amount to
requiring that the correlator assigned to a surface is invariant under the action of 
the mapping class group of the surface and that upon sewing of surfaces, correlators
are mapped to correlators \cite{fffs3}. (In addition, a non-degeneracy requirement 
must be imposed on the two-point correlator on the sphere.) An insight on which the 
present paper builds is that these requirements can be implemented with the help of the  
structural morphisms of the coends that in the construction of \cite{lyub11}
afford the sewing of spaces of conformal blocks.

The notion of a modular finite ribbon category is recalled in Section \ref{ss:cat-coend};
the category $H$-mod of finite-dimensional modules over any finite-dimensional factorizable
ribbon Hopf algebra $H$ belongs to this class of categories \cite{fuSs3}.
Modular Frobenius algebras are introduced in Definition \ref{def:modFrob}; it is known
\cite{fuSs3} that in case \D\ is the enveloping category $\C \boti \C^{\rm rev}$
of a category $\C \eq H$-mod of the type just mentioned, every ribbon automorphism
of the identity functor of \C\ gives rise to a modular Frobenius algebra in \D.
We expect that this continues to hold for the enveloping category of any 
modular finite ribbon category \C, but the methods of \cite{fuSs3} which are based
on Hopf algebra technology are insufficient to show this. Also, even when restricting to
the case $\C \eq H$-mod it is not known whether those methods can be used to address the
compatibility of correlators with sewing, while they do allow one \cite{fuSs5}
to establish invariance under the action of mapping class groups.
For \D\ the enveloping category of a \emph{semisimple} mo\-du\-lar finite ribbon category
\C, modular Frobenius algebras in \D\ are obtained by applying a center construction
\cite{ffrs,davy20} to the special
symmetric Frobenius algebras in \C\ that were considered in \cite{fuRs4,fjfrs}.
In the present paper we rely on finiteness properties that are shared by these
examples, but we neither have to require semisimplicity nor equivalence to the
representation category of a Hopf algebra.

For the construction in \cite{fuRs4,fjfrs} it is necessary to invoke the full-fledged
three-dimen\-si\-o\-nal topological field theory obtained by the Re\-she\-tikhin-Turaev 
surgery construction. The approach taken in the present paper bypasses this by making use 
of the following ingredients:
First, exploiting a variant of the Lego-Teichm\"uller game \cite{haTh,baKir} allows 
us to express the consistency constraints in terms of a small set of basic correlators;
second, we consider correlators for the full bulk state space $F$ rather than separately
for subquotients of $F$ (in the non-semisimple case we are largely forced to do so, but
this proves to be advantageous even in the semisimple case);
and third, the insight that the sewing relations among such correlators
can be neatly formulated with the help of dinatural transformations for suitable coends.
An evident question to ask is whether also the present construction can be naturally
embedded in a framework provided by three-dimensional topological field theory. It 
is tempting to speculate that this can be achieved in a `holographic approach', using 
topological field theory on manifolds with boundaries. To this end it will be sufficient 
that the theory can be defined on three-manifolds that are sufficiently nicely behaved 
(in particular on cylinders over two-manifolds); this might be viable even when the 
underlying modular category \D\ is not semisimple. 
Pursuing such speculations is, however, beyond the scope of the present paper.

\medskip
 
Before giving more details, let us put our results into context. While quantum field 
theory is studied in many different frameworks, a minimal consensus would be close to
the following: To establish the existence of a quantum field theory or, at least, of the 
subsector of bulk fields of the theory, we must specify a vector space of such fields
as well as, for a suitable category of manifolds with marked points, correlators of
those fields. The correlators are required to be local in a suitable sense,
and correlators on different manifolds must fit together.

This is far too vague for being implementable in practice, but for classes of theories 
that enjoy suitable finiteness properties and share certain symmetries the situation is
more amenable. Symmetries strongly constrain the possible correlators: they give rise to
differential equations for them, so-called Ward identities \cite[Sect.\,3]{bepz}.
The spaces of solutions to those equations -- called spaces of conformal blocks in the 
case of conformal field theories -- are thus spaces of candidates for correlators. 
In the present paper we impose the finiteness condition that the spaces of solutions 
to the Ward identities, at any genus and for any number of field insertions, are 
finite-di\-men\-sional.
The correlators are particular vectors in these spaces. The relevant locality requirements 
and relations between correlators on different manifolds then strongly restrict the possible
choices of such vectors.

As we show in this paper, this idea can be fully realized in a large class of
two-dimensional conformal field theories, including non-semisimple theories. We do not 
work with vector spaces obtained as solutions to differential equations (but, prompted 
by that interpretation, still use the term monodromy data to refer to the corresponding
information). Instead, we obtain an equivalent system of vector spaces algebraically,
as morphism spaces of a suitable \ko-linear monoidal category \D, which is possible if
\D\ is a modular finite ribbon category. This class of categories includes all 
semisimple modular tensor categories, such as those for the rational conformal field 
theories of Wess-Zumino-Witten type, for which the differential equations obeyed by the
conformal blocks are known explicitly and include in particular the equation expressing 
the flatness of the Knizhnik-Zamolodchikov connection (see \cite{ETfk} for a review).
For our construction \D\ does, however, not have to be semisimple, and thus our
analysis covers theories well beyond the subclass of rational conformal field theories.
Also, albeit for us \D\ matters only as a category endowed with appropriate structure,
it is adequate to think of \D\ as the representation category of an algebraic structure
-- a conformal vertex algebra, or a conformal net of observables -- that formalizes 
the physical notion of a chiral algebra of a two-dimensional conformal field theory. 
It has been known for a while that rational conformal field theories provide
examples for semisimple modular finite ribbon categories \cite{huan21}.
On the other hand, beyond rational CFT much less is known. Specifically, while
there has been extensive work (see e.g.\ \cite{fgst4,huan29,tswo})
on classes of non-semisimple conformal field theories with vertex operator algebras that
are expected to give rise to modular finite ribbon categories, so far the 
presence of such a category has been fully established only for the
so-called symplectic fermion model \cite{gaRu2}.
The space $F$ of bulk fields we are looking for carries a representation of the chiral
algebra (or of the tensor product of two copies of the chiral algebra,
which in physics are sometimes called left- and right-movers, respectively) and is thus 
an object of the category \D. In this paper we identify the additional structure 
on $F$ that is necessary and sufficient for obtaining a consistent system of correlators.
The existence of such structure imposes restrictions on the category \D.   

\medskip

Let us now describe the contents of this paper more explicitly. We consider a symmetric 
monoidal category \Surf\ of oriented
smooth surfaces with boundary, with the monoidal structure given by disjoint union.
As morphisms of \Surf\ we take orientation preserving diffeomorphisms 
combined with sewings $\surf \,{\to}\,
\Sqcup\surf$, i.e.\ (compare Figure \ref{fig1}(ii) below) with gluings
of surfaces via identification of boundary circles. 
The boundary circles are the recipients of insertions of (incoming or outgoing)
bulk fields, which are described as an object $F$ of the category \D.
To each surface $E \iN \Surf$ we assign, in a two-step procedure, a 
bulk field correlator as a vector $\vs(\surf)$ in a vector space $\Bl(\surf)$.
The vector spaces $\Bl(\surf)$, which are the conformal block spaces, are obtained as
morphism spaces of the \ko-linear category \D; this first step is well known \cite{baKir,lyub11}, 
albeit (see Section \ref{sec2}) it still needs to be adapted to fit our purposes.

Each space $\Bl(\surf)$ carries a projective representation of the mapping class group of 
$\surf$, and the collection of these spaces has locality properties. The latter are encoded 
in linear maps between spaces of conformal blocks assigned to surfaces of different topology
that are related by sewing. The novel main part of our construction is then to find a system
of vectors $\vs(\surf) \iN \Bl(\surf)$ that are invariant under the mapping class group 
action and such that sewing $\surf \,{\to}\,
\Sqcup\surf$ transports $\vs(\surf)$ to
$\vs(\Sqcup\surf)$ -- provided that such a set of vectors exists at all, a requirement
that puts restrictions on the allowed categories \D\ and bulk objects $F$ in \D.
A crucial idea of the construction is to start from simple building pieces
not only for the spaces of conformal blocks, but correspondingly also for vectors in 
those spaces as building blocks of the correlators.

Our choice of morphisms of the category \Surf\ allows us
to treat compatibility with sewing and mapping class group invariance 
on the same footing. Concretely, we can describe the conformal blocks $\Bl(\surf)$ as 
the vector spaces that a symmetric monoidal functor $\Bl$ assigns to the objects of $\Surf$,
and a consistent system $\{\vs(\surf)\}$ of bulk field correlators as a 
monoidal natural transformation 
  $$
  \vs: \quad \One \Longrightarrow \Bl
  $$
between a constant functor $\One$, with value the ground field,
from \Surf\ to \Vect\ and the block functor $\Bl$.

At the level of the vector spaces of conformal blocks, locality can be implemented by 
using pair-of-pants decompositions of the objects $\surf\iN\Surf$. Moreover, such
decompositions also allow one to get control over the action of the mapping class group,
at the expense of endowing the surfaces with further structure, namely a certain embedded 
graph called a \emph{marking} \cite{baKir}.
With these extra structures of pair-of-pants decompositions and markings, we get 
a symmetric monoidal category \mSurf\ of \emph{marked surfaces} together with a
forgetful functor $U\!\colon \mSurf \To \Surf$. 
For the construction of the functor $\Bl$ it proves to be convenient to 
make explicit use of the extra structure of \mSurf, whereby at first one 
arrives at a similar symmetric monoidal functor $\mBl\colon \mSurf \To \Vect$.
(Actually, to account for the framing anomaly \cite{witt50}, one must work with 
central extensions of the categories \Surf\ and \mSurf, see Section \ref{ss:cext};
for brevity we suppress this issue in this introductory exposition.)

   \medskip

Our strategy for proving our main result is now the following: First, we construct the 
monoidal functor $\Bl$ as a right Kan extension (see e.g.\ \cite[Ch.\,1]{RIeh})
  $$ 
  \begin{tikzcd}
  \mSurf \ar{rr}{\mBl}[name=endNatTraf,below]{} \ar{d}[swap] {U} &~& \Vect
  \\
  \Surf \ar[dashed]{urr}[swap] {\Bl}
  \ar[Rightarrow,xshift=-5pt,to path=-- (endNatTraf) \tikztonodes]{}{}
  \end{tikzcd}
  $$ 
of $\mBl$ along the forgetful functor $U$. Hereby we can reduce the existence of a 
monoidal natural transformation $\vs\colon \One \,{\Rightarrow}\, \Bl$
to the existence of an analogous natural transformation
  $$
  \mvs: \quad \mOne \Longrightarrow \mBl
  $$
from a constant functor $\mOne\colon \mSurf \To \Vect$ to the functor $\mBl$. Second, we
establish necessary and sufficient conditions for the existence of such a natural transformation
$\mvs$. The latter is achieved with the help of the description of the morphisms of $\mSurf$
through generators and relations (a variant of the so-called Lego-Teichm\"uller game),
which in particular allows us to express the required restrictions on the bulk object
$F$ in terms of a triple of vectors in specific morphism spaces involving $F$, namely
in $\HomD(\one,F^{\otimes 3})$, $\HomD(F,\one)$ and $\HomD(F,F^\vee)$, respectively.
It turns out that once the problem is reformulated in terms of this triple of morphisms,
it can be solved, and in particular the required
additional structure of $F$ is expressible in terms of the chosen triple of morphisms.
It is in fact an old idea in conformal field theory that correlators on arbitrary surfaces 
are determined by a small set of correlators in low genus, subject to only a few
consistency constraints. Our results make this idea precise for theories that are 
not required to be semisimple.

\medskip

The rest of this paper is organized as follows.
Section \ref{sec2} is preparatory; it collects pertinent information about the 
Lego-Teichm\"uller game \cite{baKir} and about the construction of conformal blocks,
adapting the results of \cite{lyub11} such that they can be combined with the framework
of \cite{baKir} that was designed for finite semisimple categories. Specifically, we 
explain the notions of (extended) surfaces and marked surfaces and the groupoid of fine 
markings on a surface, and describe the conformal blocks for a finite ribbon category
\D\ as \ko-linear functors from tensor powers of \D\ to \Vect.
In Section \ref{sec3} we combine, for a choice of object $F \iN \D$, the conformal block 
functors for all surfaces to a monoidal functor $\BlF$ from (a central extension of) a 
symmetric monoidal category \Surf\ of surfaces with values in vector spaces.
Invariance of the correlators under the action of mapping class groups and compatibility
with sewing are then formulated as a monoidal natural transformation from a constant
functor $\One$ to the functor $\BlF$.  For arriving at this formulation we first work
with a monoidal functor $\mBlF$ having as domain a
category \mSurf\ of marked surfaces and then obtain $\BlF$ via a Kan extension.
Finally in Section \ref{sec4} we show that, subject to a non-degeneracy condition,
the existence of a monoidal natural transformation $\One\,{\Rightarrow}\,\BlF$ is
equivalent to $F$ having the properties stated in 
Proposition \ref{prop-0} and Theorem \ref{thm-g}.


\section{Conformal blocks}\label{sec2}

In this section we collect pertinent background information on marked surfaces,
finite ribbon categories and conformal blocks.
The expert reader may wish to have just a quick glance at this part, e.g.\ at
Definitions \ref{def:xtsurf}, \ref{def:standardmarking} and \ref{def:fine},
at the figures \ref{fig1} and \ref{fig2}, at the list of elementary moves 
(M1)\,--\,(M5) in Section \ref{sec:fine}, and at the formulas \erf{eqdef:mBl-g}
and \erf{defBlgn} for the conformal blocks. It should be appreciated that
\erf{eqdef:mBl-g} and \erf{defBlgn} do not require semisimplicity.

\subsection{Marked surfaces}

By a \emph{surface} $\surf$ we mean a compact oriented smooth two-dimensional manifold,
possibly disconnected and possibly with boundary. We endow each connected component
$\alpha$ of the boundary $\partial\surf$ with the structure
of an oriented 1-manifold. If the 1-orientation on $\alpha$ coincides with the one induced
by the 2-ori\-en\-tation of $\surf$, we refer to $\alpha$ as an \emph{outgoing} boundary
circle, otherwise as an \emph{incoming} one, and denote the union of all outgoing and
all incoming boundary circles by $\surfo$ and $\surfi$. respectively.
We call $(\surf,\surfi,\surfo)$ a \emph{surface with oriented boundary}.
The surfaces of our interest are endowed with additional structure on the boundary
(compare \Cite{Def.\,2.1}{baKir} for surfaces with non-oriented boundary circles):

\begin{defi}\label{def:xtsurf}
An \emph{extended surface} $\surf \eq (\surf,\surfi,\surfo,\{p_\alpha\})$ is
a surface $(\surf,\surfi,\surfo)$ with oriented boundary
together with a marked point $p_\alpha$ on each connected component of $\partial\surf$.
\end{defi}

\begin{defi}\label{def:maps}
The \emph{mapping class group} \maps\ of an extended surface $\surf$ is the
group of homotopy classes of orientation preserving diffeomorphisms $\surf \To \surf$ 
that map $\surfi$ to $\surfi$ (and thus $\surfo$ to $\surfo$) 
and map marked points to marked points.
\end{defi}

The \emph{genus} $g(\surf)$ of an extended surface is the genus of the closed 
surface that is obtained from $\surf$ by gluing disks to all boundary components.
An extended surface can be glued, or sewn: for $(\alpha,\beta) \iN \piodsi \Times \piodso$,
the sewn surface $\Sqcup_{\alpha,\beta}\surf$ is the extended surface obtained by 
identifying the boundary components $\alpha$ and $\beta$ in such a way that their 
orientations match and that their marked points get identified.
(In principle we must glue along collars for $\Sqcup_{\alpha,\beta}\surf$ to be smooth;
but for our purposes this is inessential, compare e.g.\ \Cite{Thm.\,1.3.12}{KOck}.)

A \emph{cut} on an extended surface $\surf$ is a smooth oriented simple closed curve in the 
interior of $\surf$ together with a distinguished point on the curve.
A \emph{cut system} \cite{haTh} on $\surf$ is a finite union
$C$ of disjoint cuts, such that each connected component of $\surf \,{\setminus}\, C$ has 
genus 0.  Given a cut system $C$, let $\oline{\surf}{C}$ denote the closed manifold obtained
from the disjoint union of all components of $\surf \,{\setminus}\, C$ by suitably
adding two copies of each cut in $C$. This manifold acquires the structure of
an extended surface by endowing the two components of $\partial \oline{\surf}{C}$
that come from a cut $c \iN C$ with the 1-ori\-en\-tation of $c$
and by taking as marked points on them the points
that come from the distinguished point of $c$. This is illustrated in Figure \ref{fig1}.
Note that cutting and sewing are inverse operations:
for any cut $c$ on $\surf$, sewing the cut surface $\oline{\surf}{\{c\}}$ 
along the two boundary components stemming from the cut $c$ gives back $\surf$.
 
\begin{center}
\begin{figure}[tb]
\begin{picture}(300,109)(0,0)
  \put(0,10)   {(i)}
  \put(0,1)    {\scalebox{.26} {\includegraphics{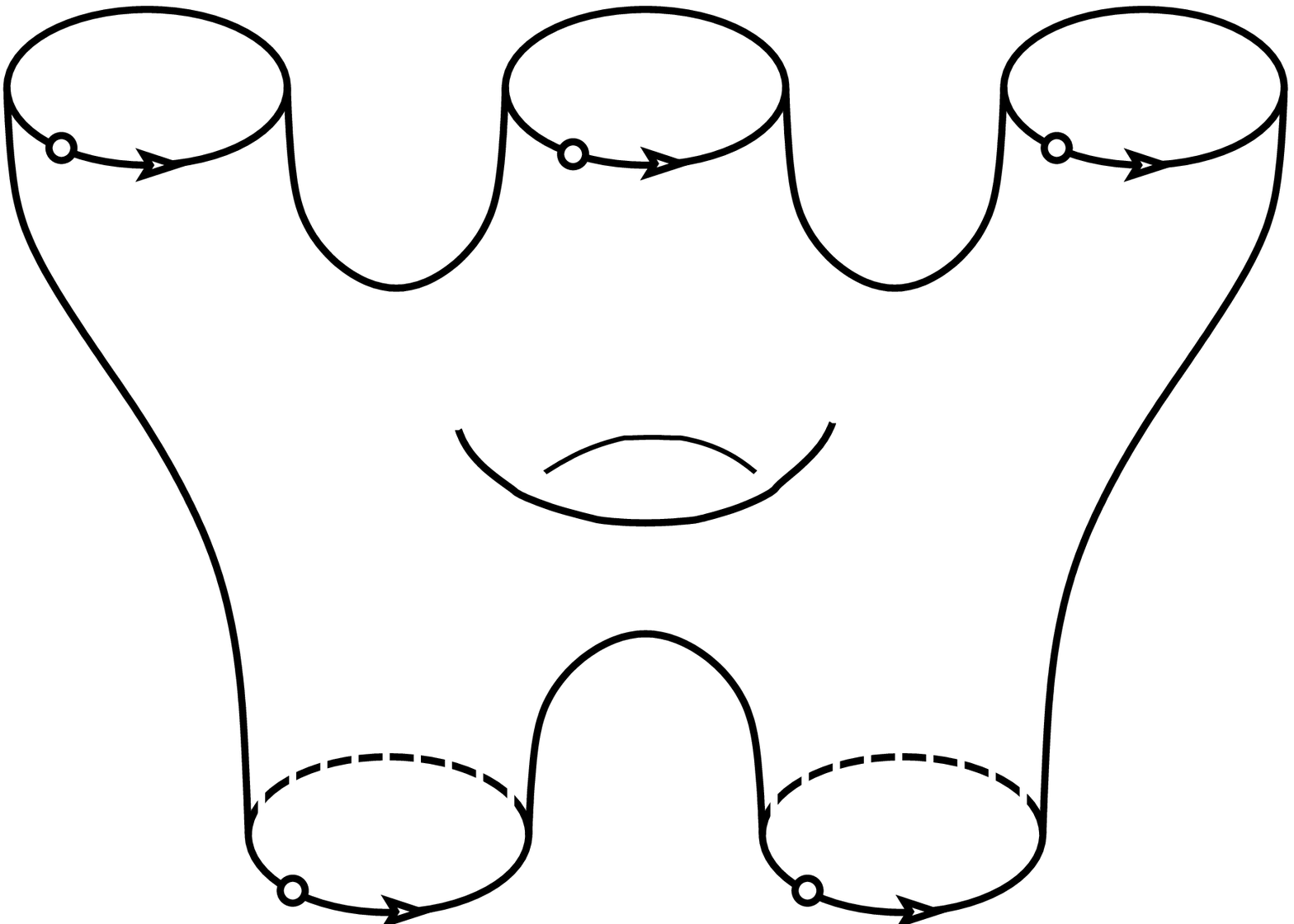}} }
  \put(166,10) {(ii)}
  \put(171,1)  {\scalebox{.26} {\includegraphics{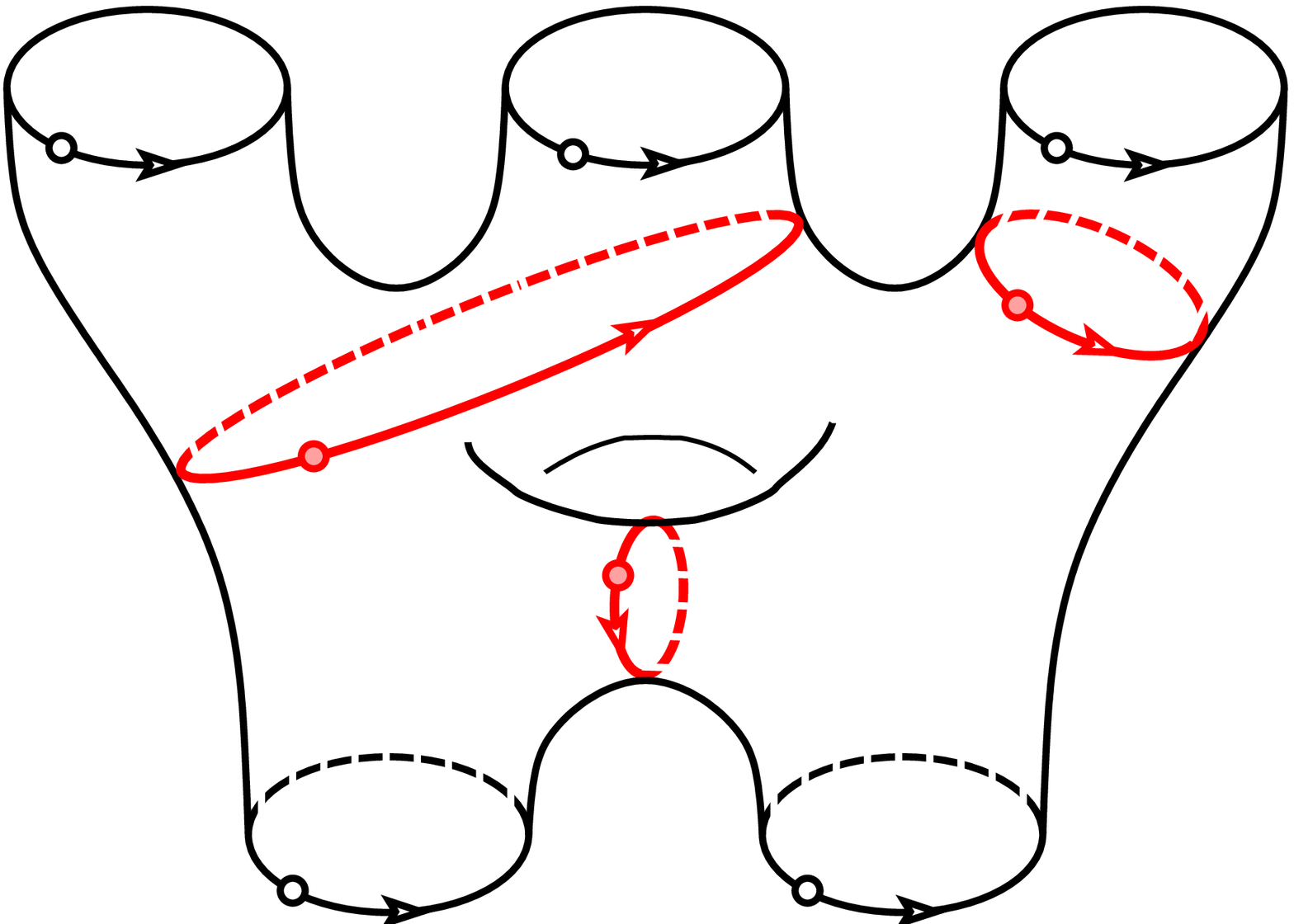}} }
  \put(305,37) {$ \xrightarrow{~~{\displaystyle\rm cut}~~} $}
  \put(303,23) {$ \xleftarrow[~~\raisebox{4pt}{sew}~~]{} $}
  \put(332,-6) {\scalebox{.26} {\includegraphics{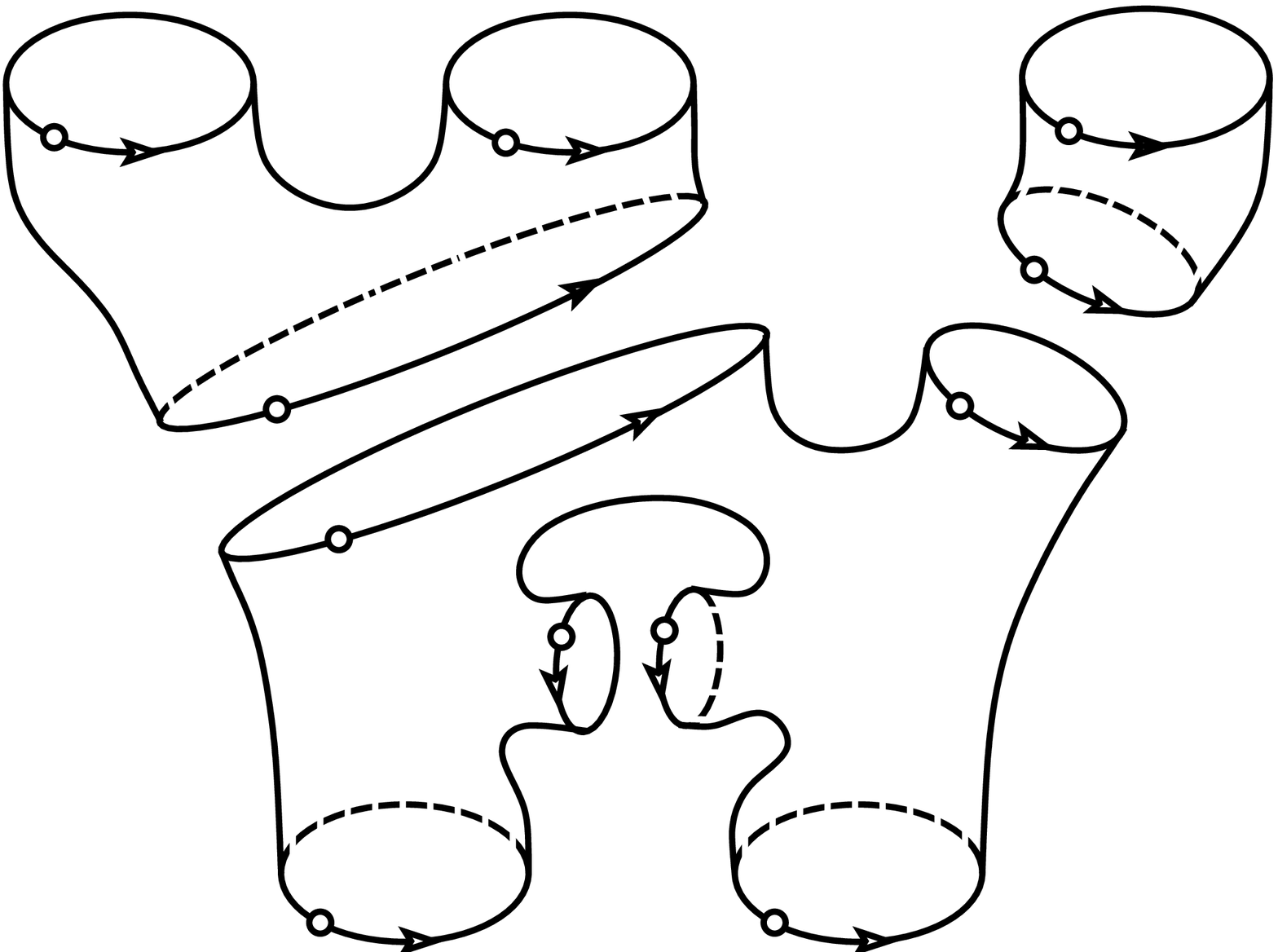}} }
\end{picture}
\caption{(i) An extended surface $\surf$ of genus 1 with three outgoing and two incoming 
  boundary circles.\,
  (ii) A (non-fine) cut system $C$ on $\surf$ (left) and
  the cut surface $\oline{\surf}{C}$, having spheres with 3, 2 and 6 boundary
  circles as connected components (right). \label{fig1}}
\end{figure}
\end{center}

We need to refine this notion of surface further. Similarly as in \Cite{Sect.\,2.3}{baKir} 
we first introduce specific reference surfaces.
For $n \iN \naturals$ and $\eps \eq (\eps_1,\eps_2,...\,,\eps_n) \iN \{\pm1{\}}^n$,
let the \emph{standard sphere} $\sSne$ be the extended surface obtained by removing
from the Riemann sphere $\complex\,{\cup}\,\{\infty\}$ with standard orientation 
$n$ open disks $D_{\alpha}$ of radius $1/3$ centered at the points $1,2,...\,,n$,
with the component $(\partial\sSne)_\alpha \eq \partial D_{\alpha}$ of the boundary
outgoing if $\eps_\alpha \eq {+}1$ and 
incoming if $\eps_\alpha \eq {-}1$, and with the marked points being
$\{ \alpha \,{-}\,\ii/3 \,|\, \alpha \eq 1,2,...\,,n \}$.
The \emph{standard marking} $\markstd_{\!n}$ on $\sSne$ is the following graph on $\sSne$:
The vertices are the marked points $\alpha\,{-}\,\ii/3$, for $\alpha \eq 1,2,,...\,,n$,
together with a vertex at $-2\ii\iN \sSne$, called the \emph{internal vertex}; the
edges are the $n$ straight lines $e_\alpha$ connecting the marked points $\alpha$ with
the internal vertex. The set of edges of $\markstd_{\!n}$ is ordered according to the
standard order on $\naturals$; the edge connecting the internal vertex to the left-most
marked point $1\,{-}\,\ii/3$ is called the \emph{distinguished edge}
and denoted by $\edist$, and $1\,{-}\,\ii/3$ is called the \emph{distinguished vertex}.
We can now give

\begin{defi} \Cite{Def.\,3.3\,\&\,Sect.\,3.6}{baKir}
\label{def:standardmarking} 
~\\[2pt]
{\rm (i)}\, A \emph{marking without cuts} on a connected extended surface $\surf$ of genus 
zero is a graph $\mark$ on $\surf$ for which there exists 
an orientation preserving diffeomorphism $\varphi\colon \surf \To \sSne$, for some
appropriate $n \iN \naturals$ and $\eps \iN \{\pm1\}^n_{}$,
such that $\mark \eq \varphi^{-1}(\markstd_{\!n})$.
\\[2pt]
{\rm (ii)}\, A \emph{marked surface} \surfcm\ is an extended surface $\surf$
endowed with the structure of a \emph{marking}, i.e.\ with a cut system $C$ and a graph 
$\mark$ that provide a marking without cuts on every connected component of $\oline\surf{C}$.
\\[2pt]
{\rm (iii)}\, Two markings $(C,\mark)$ and $(C',\mark')$ on $\surf$ are called
\emph{isotopic} iff there exists an isotopy $f\colon \surf\Times[0,1] \To \surf$ such that 
$(f(C,t),$\linebreak[0]$f(\mark,t))$ furnishes a marking on $\surf$ for every $t \iN [0,1]$
as well as $f(-,0) \eq \id_\surf$ and $f(C,1) \eq C'$, $f(\mark,1) \eq \mark'$.
\\[2pt]
{\rm (iv)}\, For $(\alpha,\beta) \iN \piodsi \Times \piodso$, the \emph{sewn surface}
$(\surf',C',\mark') \eq \Sqcup_{\alpha,\beta}(\surf,C,\mark)$ is the mar\-ked surface
with underlying surface $\surf' \eq \Sqcup_{\alpha,\beta}\surf$ whose cut system $C'$
is given by the union of the cut system $C$ of $(\surf,C,\mark)$ and the image of
$\alpha$ and $\beta$, and whose graph $\mark'$ is
obtained from $\mark$ by gluing and by taking the image of the marked point on
$\alpha$ and $\beta$ as an additional vertex (compare \Cite{Fig.\,6}{baKir}).
\end{defi}

In the sequel we often suppress the cut system in our notation and just write
\surfm\ for a marked surface, and refer to an isotopy class of markings just as a marking.
 %
It will be sufficient to work with the subclass of \emph{fine} markings:

\begin{defi}\label{def:fine}
A \emph{\fine\ cut system} on a surface $\surf$ is a cut system $C$ for which
every connected component of $\oline\surf{C}$ is a sphere with at most three holes.
A \emph{\fine\ marking} $(C,\mark)$ of $\surf$ is a marking for which 
the cut system $C$ is \fine.
\end{defi}

As an illustration, Figure \ref{fig2} 
shows a fine cut system and a fine marking on the surface from Figure \ref{fig1}.

\begin{center}
\begin{figure}[tbh]
\begin{picture}(240,91)(0,0)
  \put(70,10)  {(i)}
  \put(70,1)   {\scalebox{.26} {\includegraphics{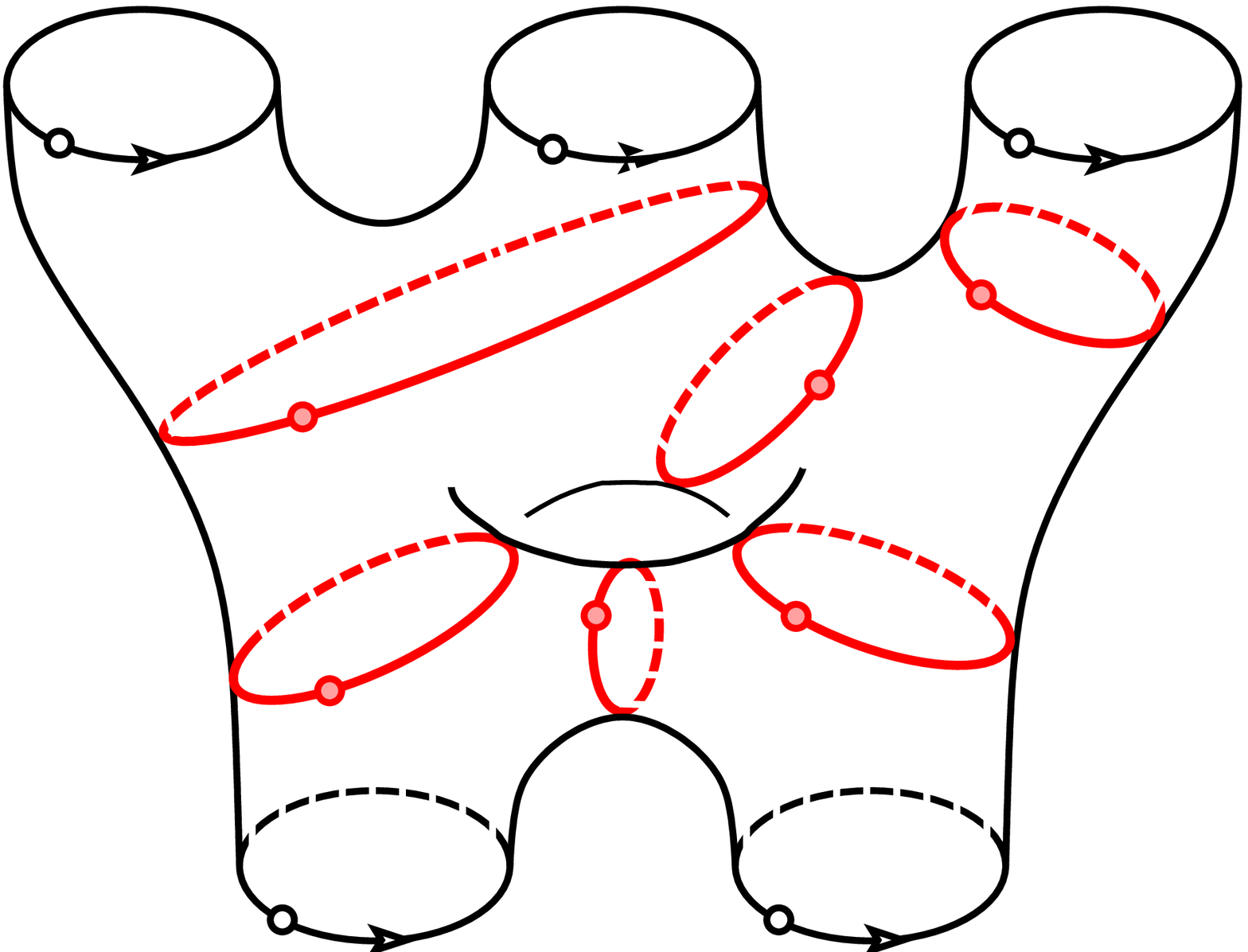}} }
  \put(266,10) {(ii)}
  \put(269,1)  {\scalebox{.26} {\includegraphics{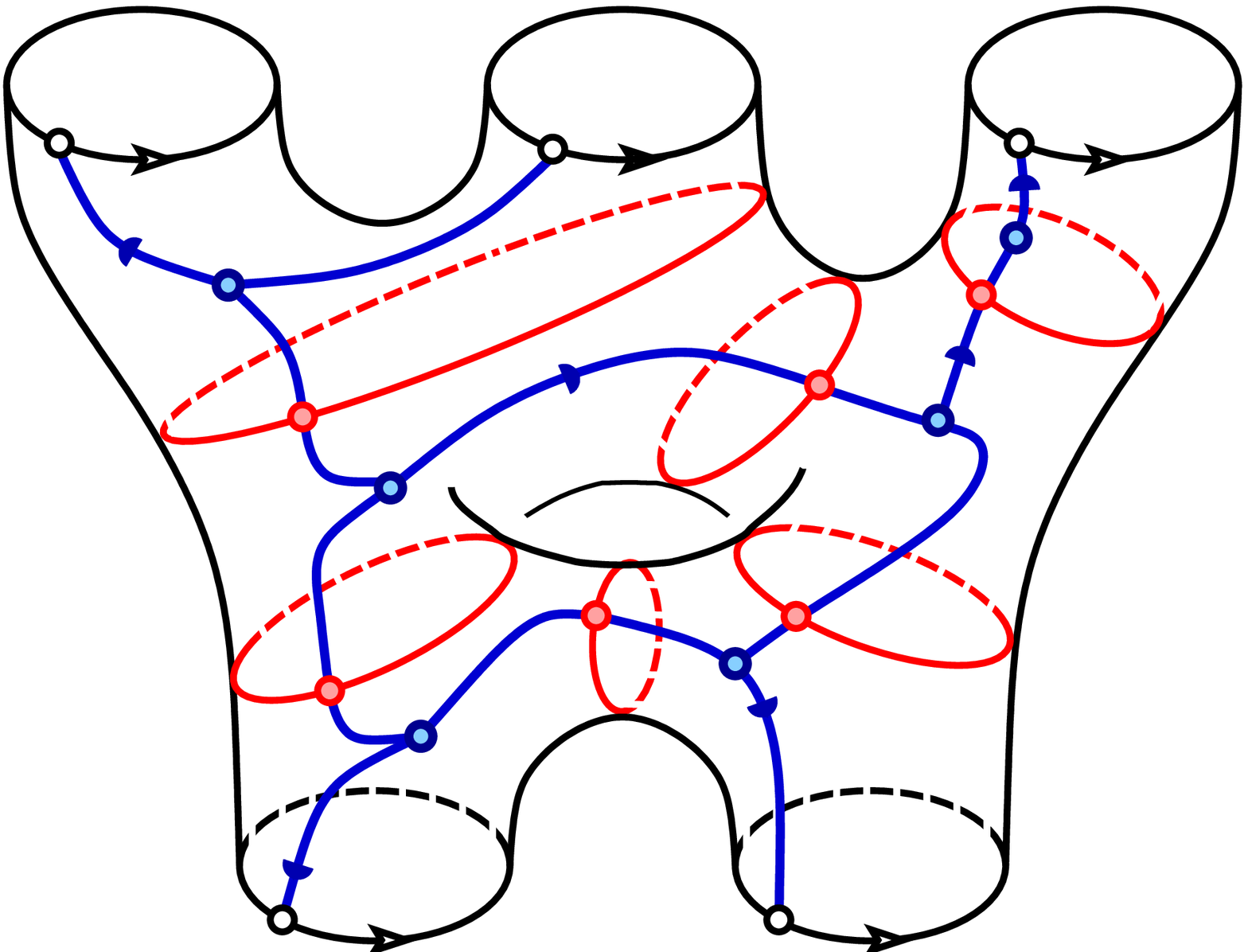}} }
\end{picture}
\caption{(i) A fine cut system $C$ on the surface $\surf$ from Figure \ref{fig1}.
  (The resulting cut surface is the disjoint union of one 2-holed and five 3-holed
  spheres.)\,
  (ii)\, A fine marking $(C,\mark)$ on $\surf$. The distinguished edge of the 
  restriction of the graph $\mark$ to each connected component of the cut surface is 
  accentuated by a small triangular flag.
  (For better readability, the 1-orientation of the cuts is suppressed.)
  \label{fig2}}
\end{figure}
\end{center}


\subsection{Fine markings}\label{sec:fine}

To any extended surface $\surf$ one can associate a groupoid describing the set of
isotopy classes of markings on $\surf$ and their relations \cite{haLS,fuGel,baKir}.
Equivalent to this groupoid is a groupoid that we denote by $\CWm(\surf)$.
The objects of $\CWm(\surf)$ are the (isotopy classes of) fine markings on $\surf$, and  
its morphisms are (classes of) finite sequences of moves that 
change a fine marking of $\surf$ to another one. A geometric realization of $\CWm(\surf)$ 
is furnished by a graph with vertices given by the objects of $\CWm(\surf)$. This graph
is connected and simply connected, and accordingly $\CWm(\surf)$ can be presented by 
generators and relations (see \Cite{Thm.\,5.1}{baKir} and \cite{bePi}).
The generators, to be called elementary moves, of $\CWm(\surf)$ and the relations 
satisfied by them 
can be taken as follows \Cite{Sect.\,5}{baKir}. There are five types of elementary moves:
 \def\leftmargini{2.88em}~\\[-1.55em]\begin{itemize}\addtolength{\itemsep}{-5pt}
  \itm{1}
The \emph{Z-move} \mZ\ of a two- or three-holed sphere $\surf$ without cuts. 
This move maps the graph $\mark$ on $\surf$ to the graph $\mark'$ that coincides with $\mark$
as an unordered graph and whose distinguished edge is the one
adjacent to the distinguished edge of $\mark$ in clockwise direction,
 \footnote{~Our convention differs from the one in \cite{baKir}, where the 
 distinguished edge is changed in counter-clockwise direction instead.} 
and which keeps the cyclic ordering of the edges.
  \itm{2}
The \emph{B-move} \mB\ of a sphere $\surf$ with three holes and without cuts 
\Cite{Fig.\,10}{baKir}. For the case that $\surf \eq \sS_3$ is the 3-holed
standard sphere with standard marking, the move $\mB$ results in a marking that 
can also be obtained by performing the following braiding diffeomorphism: move the
boundary circles centered at $1 \iN \complex$ and $2 \iN \complex$ by an angle $\pi$
clockwise around $\frac32\iN \complex$ such that each of them is mapped to the
previous position of the other one, while the third boundary circle is kept in place.
 %
  \itm{3}
The \emph{F-move} \mF\ of a sphere $\surf \eq (\surf,\{c\},\mark)$ with at most three
holes, with a single cut $c$ such that one of the two components of $\oline\surf{\{c\}}$ 
is a sphere with one or two holes, and a graph $\mark$ such that the edges of $\mark$
ending at the distinguished point $p_c \iN c$ are the distinguished edge on
one of the components of $\oline\surf{\{c\}}$ and the `last' edge on the other component. 
$\mF \,{\equiv}\, \mF_c$ removes the cut $c$ and contracts the edges of $\mark$ ending at
$p_c$ to a point (compare \Cite{Fig.\,9}{baKir}).
  \itm{4}
The \emph{A-move} \mA\ of a sphere $\surf$ with four holes and a single cut $c$ and a
graph $\mark$ such that cutting along $c$ gives two 3-holed spheres and that 
the edges of $\mark$ ending at the distinguished point of $c$
are the `last' edges of the graphs on both components of $\oline\surf{\{c\}}$.
$\mA \,{\equiv}\, \mA_c \,{\equiv}\, \mA_{c,c'}$ replaces $c$ by another cut $c'$ 
not isotopic to $c$ such that $\oline\surf{\{c'\}}$ consists of two three-holed spheres;
for details see \Cite{Fig.\,20}{baKir}.
  \itm{5}
The \emph{S-move} \mS\ of a one-holed torus $T$ with a single cut. 
$\mS\,{\equiv}\,\mS_{c_1,c_2}$ maps the marking
$(\{c_1\},\mark_{\!1})$ of $T$ to $(\{c_2\},\mark_{\!2})$, with $\{c_1,c_2\}$
a symplectic homology basis of $H_1(T,\zet)$, and with 
$\mark_{\!1}$ and $\mark_{\!2}$ graphs having a common single vertex in the
interior of $T$ and two edges $\{\edist,e'_1\}$ and $\{\edist,e'_2\}$,
respectively. The common distinguished edge $\edist$ connects
the interior vertex with the boundary circle, while $e'_1$ and $e'_2$ are loops
homotopic to $c_2$ and $c_1$, respectively \Cite{Fig.\,16}{baKir}.
\end{itemize}

\noindent
And the relations among these moves can be taken to be the following
equalities between isotopy classes of compositions of moves
(throughout, as in \cite{baKir}, some easily 
reconstructed intermediate Z-moves are omitted in order to improve readability):
 \def\leftmargini{2.98em}~\\[-1.11em]\begin{itemize}\addtolength{\itemsep}{-5pt}
  \ite{1}
Commutativity of moves in different connected components of \surf.
  \ite{2}
The cylinder axiom: Given the standard cylinder $S\eq (\sS_{1,1},\emptyset,\markstd_{\!1,1})$
with standard marking and with one incoming and one outgoing boundary component, 
a gluing $\Sqcup_{\gamma,\beta}$ of $S$ to a surface $(\surf,C,\mark)$ (with 
$\beta \iN \pi_0(\partial S)$ and $\gamma \iN \pi_0(\partial\surf)$)
and any move $\mathrm m\colon (\surf,\mark) \To (\surf,\mark')$, one has
  $$
  \psi \circ \mF_\gamma \circ (\mathrm m \,{\cup_{\gamma,\beta}}\, \id)
  = \mathrm m \circ \psi \circ \mF_\gamma
  $$
with $\psi$ a morphism which amounts to a compression
of $\surfm \,{\cup_{\gamma,\beta}}\, (\sS_{1,1},\emptyset,\markstd_{\!1,1})$ to \surfm\
(compare \Cite{Fig.\,12}{baKir}).
  \ite{3}
For any sphere with $n \iN \{2,3\}$ holes, the Z-move obeys $\mZ^n_{} \eq \id$.
  \ite{4}
Compatibility of F- and Z-moves: For an F-move of a surface $\surf$ with a single cut $c$ 
such that $\oline\surf{\{c\}} \eq \surf_1 \,{\sqcup}\, \surf_2$ with $\surf_1$
the component containing the distinguished edge ending on $c$ and $n_1 \,{:=}\, |\piode|$
one has $\mZ^{1-n_1} \cir \mF \eq \mF \cir (\mZ \,{\sqcup}\, \mZ^{-1})$.
  \ite{5}
Compatibility of B- and Z-moves: For $(\surf,\emptyset,\mark)$ a
cylinder with a marking without cuts one has $ \mZ \cir \mB \eq \mB \cir \mZ $.
  \ite{6}
Commutativity of F-moves involving a cylinder: If $(C,\mark)$ is a
\fine\ marking on $\surf$ with cut system $C \eq \{c,d\}$ such that one of the 
components of $\oline\surf{C}$ is a cylinder, then $\mF_c \cir \mF_d \eq \mF_d \cir \mF_c$.
  \ite{7}
Involutivity of the A-move: $\mA^2 \eq \id$.
  \ite{8}
The triangle axiom: For a marking of a $3$-holed sphere with cut system $C \eq \{c,d\}$
such that when cutting along $d$ one of the resulting connected components is 
a one-holed sphere, one has $\mF_{c'} \cir \mF_d \cir \mA \eq \mF_c \cir \mF_d$,
where $c'$ is the cut created by the A-move (for details see \Cite{Fig.\,29}{baKir}).
  \ite{9}
The pentagon axiom for the A-move, an analogue of the pentagon identity for 
the associator of a monoidal category: For a fine marking $(C,\mark)$ of a $5$-holed 
sphere with $C \eq \{c,d\}$ and $\mark$ a multiperipheral graph, one has
$ \mA_{c'} \cir \mA_d \cir \mA_c \eq \mA_c \cir \mA_d $, where
$c'$ is the cut that is created by the A-move $\mA_c$ (compare \Cite{Fig.\,30}{baKir}).
  \ite{10}
Two hexagon axioms for the B- and A-moves, analogues of the hexagon identities for 
the braiding and associator of a braided monoidal category: For $S$ 
a sphere with four holes, labeled $\alpha$, $\beta$, $\gamma$ and $\delta$, and
a fine marking $(\{c\},\mark)$ on $S$ 
with $\mark$ a multiperipheral graph whose distinguished edges end on the 
boundary component $\alpha$ and on the cut $c$, respectively, one has
$ \mB_{\alpha,\gamma} \cir \mA_c \cir \mB_{\alpha,\beta} \eq \mA_{c'} \cir \mB_{\alpha,c'}
\cir \mA_c$ together with the equality obtained by replacing all B-moves by their
inverses, where again $c'$ is the cut created by the A-move $\mA_c$
(compare \Cite{Fig.\,31}{baKir}).
  \ite{11}
The first of the two \slz-relations for a $1$-holed torus $T$: For any marking with 
a single cut $c$ on $T$ one has $\mB_{c,\alpha} \cir \mZ \eq \mS^2$,
where $\alpha$ is the boundary circle of $T$ and $\mT_c$ is
the \emph{Dehn move} around the cut $c$, i.e.\
is a specific composition \Cite{Ex.\,4.15\,\&\,4.17}{baKir} of B-, Z- and F-moves.
  \ite{12}
The second of the \slz-relations for a $1$-holed torus $T$: With the same notations as
in (W11) one has $(\mS \cir \mT_c)^3_{} \eq \mS^2$.
  \ite{13}
For a $2$-holed torus $T$ with boundary circles $\alpha$ and $\beta$ and a specific
marking on $T$ with cut system consisting of two cuts $c$ and $d$ (for details see 
\Cite{App.\,B}{baKir}), the equality
$ \mZ \cir \mB_{\alpha,\beta}^{} \cir \mA_{c,c'}^{} \cir \mA_{d,d'}^{}
\eq \mS^{-1}_{c'',d'} \cir \mA_{d'',c'}^{} \cir \mT_{c''}^{} \cir \mT_{d''}^{-1}
\cir \mA_{d,d''}^{} \cir \mS^{-1}_{c,c''} $.
\end{itemize}


\subsection{Finite ribbon categories and coends}\label{ss:cat-coend}

Let \ko\ denote an algebraically closed field, and \Vect\ the category of \findim\
vector spaces over \ko. A finite tensor category \cite{etos} is a \ko-linear abelian
rigid monoidal category such that all morphism spaces are \findim\ over \ko,
there are up to isomorphism finitely many simple objects, each of them has a 
projective cover, every object has finite length, and the monoidal unit $\one$ is simple.
A ribbon category is a rigid braided monoidal category endowed with a
compatible twist (or balancing) or, equivalently, a braided pivotal category.
By a \emph{finite ribbon category} we mean a finite tensor category that is also ribbon.
 
Given a functor $F\colon\, \mathcal A \,{\times}\, \mathcal A^{\mathrm{op}} \To \mathcal B$,
a dinatural transformation from $F$ to an object $B \iN \mathcal B$ is a family of morphisms
$\varphi_X\colon F(X,X)\To B$ for $X \iN \mathcal A$ such that
$\varphi_Y \cir F(Y,g) \eq \varphi_X \cir F(g,X)$ for all $g\,{\in}\,\Hom_{\mathcal A}(X,Y)$.
A \emph{coend} $(C,\iota)$ for $F$ is an object $C \iN \mathcal B$
with a dinatural transformation $\iota$ that is universal among all dinatural 
transformations from $F$ to an object of $\mathcal B$ in the sense that for any such
family $\varphi$ there exists a unique morphism
$\kappa$ obeying $\varphi_X \,{=}\, \kappa \,{\circ}\, \iota_X$ for all $X \iN \mathcal A$.
Both the pair $(C,\iota)$ and the object $C$ are denoted by $\int^{X\in\mathcal A}\! F(X,X)$.
For any finite tensor category \D\ the coend
  \be
  K := \int^{X\in\D}\!\! X \oti X^\vee
  \labl{K}
exists and has a natural structure of a coalgebra in \D. If \D\ is in addition braided, 
then \cite{lyub8} $K$ carries a 
natural structure of a Hopf algebra  endowed with a non-zero left integral as well as 
with a Hopf pairing $\varpi_K$. A finite ribbon category \D\ is called \emph{modular} 
iff the pairing $\varpi_K$ is non-degenerate \cite{KEly}. Modularity of \D\ is in particular
equivalent \cite{shimi10} to the property that the functor from the enveloping category 
$\D \boti \D^{\rm rev}$ to the center $\mathcal Z(\D)$ that maps $X \boti Y$ to $X\oti Y$ 
endowed with half-braiding
$ (c^{}_{X,-} \,{\otimes}\, \id_Y^{}) \,{\circ}\, (\id_X^{} \,{\otimes}\, c^{-1}_{-,Y}) $
is a braided equivalence.
If \D\ is modular, then the integral of $K$ is two-sided.
A semisimple modular finite ribbon category is the same as a modular tensor category
in the conventional (see e.g.\ \cite{BAki}) sense.

Various other coends exist in the situation of our interest as well.
For \ko-linear categories we have \Cite{Lemma\,B.1}{lyub11} 
  \be
  \int^{Y\in\D}\!\! G(Y) \otik \HomD(Y,U) = G(U)
  \labl{deltaHom}
for any left exact \ko-linear functor $G\colon \D\To\Vect$.  The component $i_Y$ of the
dinatural family of this coend is the linear map $g \oti h \,{\mapsto}\, G(h) (g)$.
An important statement about coends, which will allow us to relate different sequences of
sewings that result in one and the same surface, is that the order of iterated coends
can be interchanged. This is known as the \emph{Fubini theorem} \cite[Ch.\,IX.7]{MAcl}
for coends: Given a functor $F\colon \mathcal A \,{\times}\, \mathcal A^{\mathrm{op}}
\,{\times}\, \mathcal B \,{\times}\, \mathcal B^{\mathrm{op}} \,{\to}\, \mathcal E$ 
for which the coends $\int^{U\in\mathcal A}\! F(U,U,Y,Z)$ and
$\int^{X\in\mathcal B}\! F(V,W,X,X)$ exist for all $Y,Z\,{\in}\,\mathcal B$ and all
$V,W\,{\in}\,\mathcal A$, respectively, there are unique isomorphisms
  \be
  \bearll \dsty
  \int^{U\in\mathcal A}\! \Big( \int^{X\in\mathcal B}\!\! F(U,U,X,X) \Big) \!\!&\dsty
  \cong \int^{U\times X\,\in\,\mathcal A\times \mathcal B}\!\! F(U,U,X,X) 
  \Nxl1 &\dsty
  \cong \int^{X\in\mathcal B}\! \Big( \int^{U\in\mathcal A}\!\! F(U,U,X,X) \Big) 
  \eear
  \labl{eq:Fubini}
of coends.

In the context of conformal field theory we deal with coends with values in categories of 
functors, since conformal blocks are functors from some Deligne power of $\D$ to \Vect.\,%
  \footnote{~Recall \cite[Sect.\,5]{deli} that the Deligne product of two finite linear
  categories $\mathcal A$ and $\mathcal B$ is a finite linear category
  $\mathcal A \boti \mathcal B$ with the universal property that every right exact bilinear
  functor from $\mathcal A \Times \mathcal B$ to any finite linear category factors uniquely
  through a right exact linear functor with domain $\mathcal A \boti \mathcal B$.}
Via the so-called parameter theorem for coends \cite[Sect.\,V.3]{MAcl},
a coend with value in the functor category $\Fun(\C,\mathcal E)$ can be understood
in terms of coends taking values in $\mathcal E$, according to
$\int^{X\in\mathcal A} G(-;X,X) \eq \big( \int^{X\in\mathcal A}\widetilde G(X,X)\, \big)\,(-)$
with $\widetilde G(\reflectbox{$?$},?) \,{:=}\, G(-;\reflectbox{$?$},?)\colon
\mathcal A \Times \mathcal A^{\mathrm{op}} \To \Fun(\C,\mathcal E)$
for $G\colon  \C \Times \mathcal A \Times \mathcal A^{\mathrm{op}} \To \mathcal E$.
Like with any limit or colimit, the precise choice of domain of a functor is relevant for
its coend; thus the coend will, in general, change when the category \D\ is replaced by a 
subcategory of \D. In the context of conformal field theory it will be essential to be
allowed to invoke the isomorphism \eqref{deltaHom} and thus to consider, along with coends
taken in a functor category $\Fun(\C,\mathcal E)$, also coends taken in the subcategory of
left exact functors from \C\ to $\mathcal E$. In such a situation we denote, following
\cite{lyub11}, a coend of the latter type by the symbol $\oint$, reserving
the symbol $\int$ for the coend over the category $\Fun(\C,\mathcal E)$ of all functors.

Two results about such `left exact coends' for functors between finite tensor categories 
will be crucial in our constructions. First, we have \Cite{Sect.\,8.2}{lyub11} 
  \be
  \oint^{X\in\D}\!\! \HomD(U,V\oti X \oti X^\vee) = \HomD(U,V\oti K) 
  \labl{oint=HomK}
functorially in $U,V \iN \D$,
with $K\iN\D$ the coend \erf{K}, and with dinatural family given by post-composition 
with $\id_V\oti \iK$, where $\iK$ is the dinatural family for the coend $K$.
And second, there is a variant of the Fubini theorem, stating that \Cite{Thm.\,B.2}{lyub11}
under analogous assumptions as for the standard Fubini theorem \eqref{eq:Fubini} for a functor
$F\colon \C \Times \mathcal A \Times \mathcal A^{\mathrm{op}} \Times \mathcal B \Times
\mathcal B^{\mathrm{op}} \,{\to}\, \mathcal E$ we have unique isomorphisms
  \be
  \bearll \dsty
  \oint^{U\in\mathcal A}\! \Big( \int^{X\in\mathcal B}\!\! F(T;U,U,X,X) \Big) \!\!&\dsty
  \cong \int^{U\times X\,\in\,\mathcal A\times \mathcal B}\!\! F(T;U,U,X,X)
  \Nxl1 &\dsty
  \cong \oint^{X\in\mathcal B}\! \Big( \int^{U\in\mathcal A}\!\! F(T;U,U,X,X) \Big)
  \eear
  \labl{eq:Fubini2}
of coends, functorial in $T \iN \C$ (for details see e.g.\ \cite[Sect.\,4]{fuSc23}).

\medskip

In the sequel, \D\ will be a finite ribbon category and, unless noted otherwise, it will
be modular.


\subsection{Conformal block functors for modular finite ribbon categories}\label{sec:mBl}

It is known \cite{lyub11} 
that, given a modular finite ribbon category \D, one can assign to any extended surface
$\surf$ with $p$ incoming and $q$ outgoing boundary circles a left-exact functor from 
the Deligne product $\D^{\boxtimes q} \,{\boxtimes}\, (\D\op)^{\boxtimes p}$
to \Vect, in such a way that the so obtained vector spaces are morphism spaces of \D.
We will think of a functor with domain $\D\op$ as a contravariant functor with domain \D,
and correspondingly work with functors
  \be
  \Bl_\surf:\quad \D^{\boxtimes (p+q)} \xrightarrow{~~} \Vect
  \labl{BlE}
with covariance properties understood. We refer to these as \emph{conformal block functors}.
For compatibility with the symmetric monoidal structure of the category of surfaces and of
\Vect\ we put $\Bl(\surf {\sqcup} \surf') \,{:=}\, \Bl(\surf) \otik \Bl(\surf')$ and
$\Bl(\emptyset) \,{:=}\, \ko$. Also, it suffices to define $\Bl_\surf$ for all ordered
$(p{+}q)$-tuples of objects of \D, corresponding to objects $X_1 \boti X_2 \boti
\cdots \boti X_{p+q}$ of $\D^{\boxtimes (p+q)}$. We think of the entry $X_\alpha$
of the tuple as labeling the boundary circle $\alpha \iN \piods$.
If \D\ is semisimple, the conformal block functors constitute part of the
three-di\-men\-si\-o\-nal topological field theory that is associated with \D\ \cite{TUra};
in the general case no three-di\-men\-si\-o\-nal topological field theory exists.

To construct the functors auxiliary structure on the surfaces is needed in intermediate
steps \cite{lyub11,lyub6}. Fine markings, as introduced in Definition \ref{def:fine},
provide such an auxiliary structure. Accordingly we consider a functor
  \be
  \mBl_\surfn:\quad \D^{\boxtimes (p+q)} \xrightarrow{~~} \Vect
  \labl{mBlE}
for every finely marked surface \surfm. Henceforth we will deal exclusively with markings
that are fine; accordingly we often refer to them just as markings.

In view of their role in quantum field theory, the functors $\mBl_\surfn$ should be 
compatible with the following idea of locality: We demand that for any cut system $C$ on 
$\surf$ the functor $\mBl_\surfn$ is expressible through the functors $\mBl_{\surf_i,\mark_i}$
for the connected components of the cut surface $\oline\surf{C}$. This implies that
the functor for any surface can be obtained by implementing the sewing of spheres that have 
at most three holes and markings without cuts.
Moreover, the functors for the latter elementary world sheets should be expressible in terms 
of the tensor product of \D\ and the Hom functor; thereby they will in particular be left exact
and allow for an interpretation in terms of intermediate states
that fit into representations of some symmetry structure.
  
Implementing sewing on conformal blocks is achieved with the help of suitable 
coend constructions \cite{lyub11} that keep us within the class of left exact functors.
In more detail, we proceed as follows.

\medskip

\noindent (1)\, \emph{Spheres with at most three holes.}
\\[2pt]
Let $(\gpq 0p{3-p},\mark)$ be a marked surface of genus zero with 3 holes, $p$ of which
are incoming, endowed with a marking without cuts. Denote by $\bar\varphi$ the cyclic 
permutation of the edges of the standard marking
on the three-holed standard sphere $S^\circ_{3;\eps}$ that is induced by the
orientation preserving diffeomorphism $\varphi\colon \gpq 0p{3-p} \To S^\circ_{3;\eps}$ 
from Definition \ref{def:standardmarking}(i). Further, write
$X_\alpha\wee \,{:=}\, X_\alpha^{}$ for $\eps_\alpha \eq {+} 1$ (i.e.\ $\alpha$ 
an outgoing boundary circle) and $X_\alpha\wee \,{:=}\, X_\alpha^\vee$ for 
$\eps_\alpha \eq {-}1$ ($\alpha$ incoming).
We then define $\mBl_{\gpq 0p{3-p},\mark}$ as the left exact functor given by
  \be
  \mBl_{\gpq 0p{3-p},\mark}(X_1,X_2,X_3) := \HomD^{}(\one, X_{\bar\varphi^{-1}(1)}\wee
  \oti X_{\bar\varphi^{-1}(2)}\wee\oti X_{\bar\varphi^{-1}(3)}\wee) \,.
  \labl{defBl03}
For spheres with $n \,{<}\, 3$ holes we define $\mBl$ analogously, with the covariant
argument of the Hom functor being a tensor product having $n$ factors.

\medskip

\noindent (2)\, \emph{Spheres with any number of holes.}
\\[2pt]
Let $\surfm$ be a connected marked surface of genus zero. Denote by $(\surf_l,\mark_l)$,
$l \eq 1,2,...\,,\ell$, the connected components of the cut surface $\oline\surf{C}$.
The marking being fine, each of the surfaces $\surf_l$ is a sphere with at most three holes.
In agreement with the general remarks above, we define a left exact functor $\mBl_\surfn$
as a suitable coend over the tensor product of the conformal block functors for the
surfaces $(\surf_l,\mark_l)$.

To give this prescription explicitly, we need additional notation.
For every cut $c_k \iN C$, label the two corresponding boundary circles of $\oline\surf{C}$
by an object $Y_k \iN \D$. Denote by $X_{l;i}$, $i\iN\{1,2,...\,,n_l\}$, the labels of the
boundary circles of the component $\surf_l$ that come from the boundary of $\surf$ and by
$\tilde Y_{l;j}$, $j\iN\{1,2,...\,,m_l\}$, the labels of those which come from cuts of 
$\surf$ (such that each of the objects $Y_k$ appears precisely twice in the list of all
$\tilde Y_{l;j}$, for two distinct values of $l$). By a slight abuse of notation write
$(X_{l;1},...\,,X_{l;n_l},\tilde Y_{l;1},...\,,\tilde Y_{l;m_l})$ for the tuple of objects
of \D\ that label the boundary circles of $\surf_l$, ordered according to the ordering 
of the edges of $\mark_{\!l}$. (For an illustration of these conventions, in the case 
of a genus-one surface, see Figure \ref{fig3}.) We then set
  \be
  \mBl_\surfn(X_{1;1},...\,,X_{\ell;n_\ell})
  := \int^{Y_1 \boxtimes \,\cdots\, \boxtimes Y_{|C|} \,\in\, \D^{\boxtimes |C|}} \!
  \bigotimes_{l=0}^\ell \mBl_{\surf_l,\mark_{\!l}}
  (X_{l;1},...\,,X_{l;n_l},\tilde Y_{l;1},...\,,\tilde Y_{l;m_l}) \,.
  \labl{eq:block2}
This indeed furnishes a left exact functor, which can be seen as follows.
By invoking the Fubini theorem \eqref{eq:Fubini}
we can rewrite the right hand side as an iterated coend
  $$
  \mBl_\surfn(X_{1;1},...\,,X_{\ell;n_\ell})
  = \int^{Y_1 \in \D}\!\!\!\! \int^{Y_2 \in \D}\!\!\! \cdots \int^{Y_{|C|} \in \D} \!
  \bigotimes_{l=0}^\ell \mBl_{\surf_l,\mark_{\!l}}
  (X_{l;1},...\,,X_{l;n_l},\tilde Y_{l;1},...\,,\tilde Y_{l;m_l}) \,.
  $$
(Up to unique natural isomorphism the ordering of the cuts in $C$ does not matter; owing
to the symmetry of \Vect, neither does the ordering of the components of $\oline\surf{C}$.)
For each of the iterated coends we can then invoke, consecutively, the property 
\erf{deltaHom} of the Hom functor, combined, if needed, with the duality of \D.
Doing so, for any marked sphere $\surfm$ we obtain a canonical isomorphism
$ \mBl_\surfn(-,...\,,-) \,{\cong}\, \HomD(\one,-\wee \oti \cdots \oti -\wee)$ of left exact 
functors. (It is thus appropriate to think of the coend prescription above as a means to
ensure compatibility of the conformal blocks at genus zero with the tensor product.)

\medskip

\noindent (3)\, \emph{Disconnected surfaces.}
\\[2pt]
The functor for the disjoint union of two marked surfaces $(\surf_1,\mark_1)$ and 
$(\surf_2,\mark_2)$ is defined to be the tensor product
  \be
  \mBl_{(\surf_1,\mark_1) \sqcup (\surf_2,\mark_2)} := 
  \mBl_{(\surf_1,\mark_1)} \otimes_\ko \mBl_{((\surf_2,\mark_2)} \,,
  \labl{eq:Bl-sqcup}
and we set $\mBl_{\emptyset} \,{:=}\, \ko$.

\begin{center}
\begin{figure}[tbh]
\begin{picture}(260,145)(0,0)
  \put(150,4) {
  \put(0,0)    {\scalebox{.26} {\includegraphics{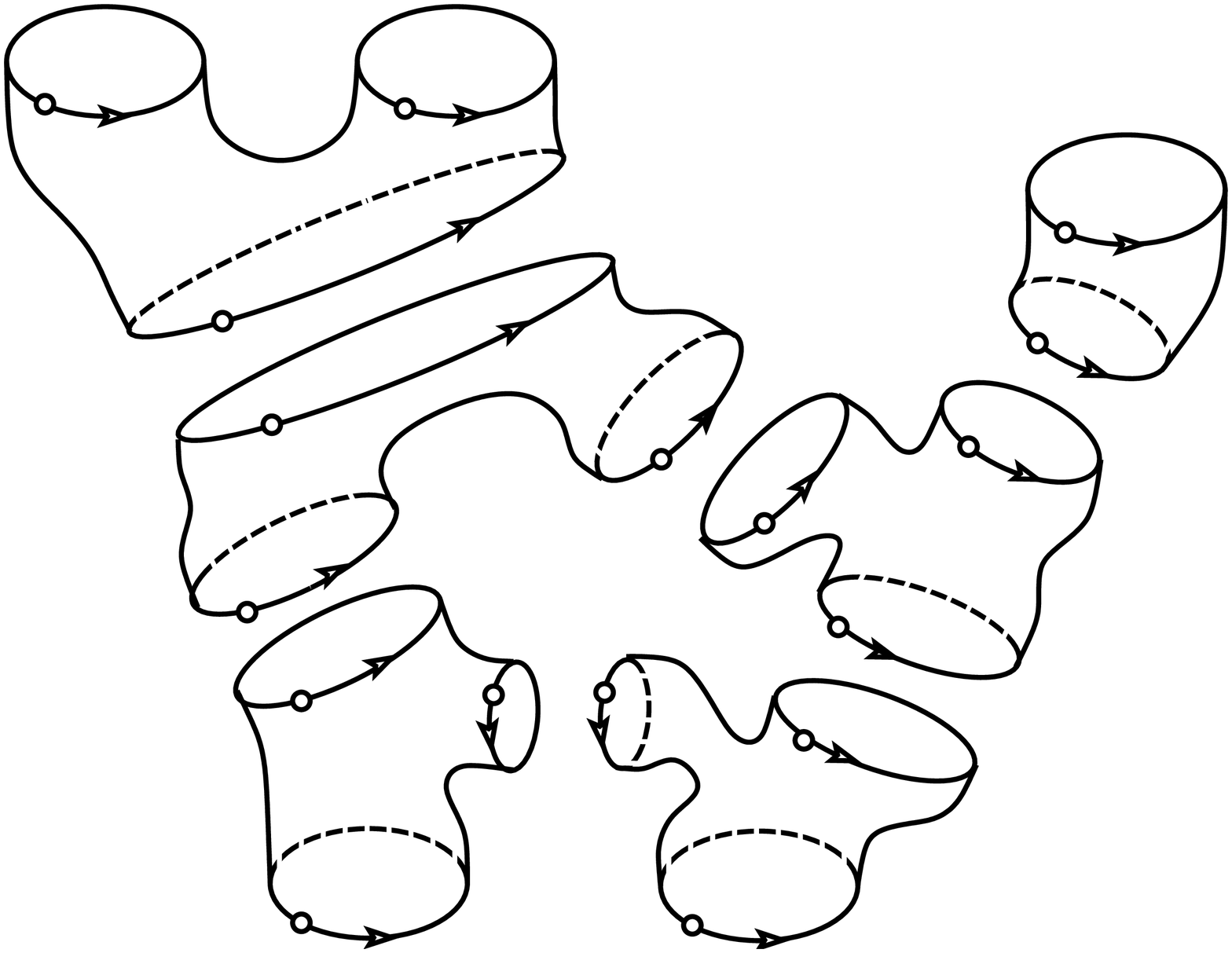}}}
  \put(5,146)  {\footnotesize $ X_{1;1} $}
  \put(59,146) {\footnotesize $ X_{1;2} $}
  \put(45,-9)  {\footnotesize $ X_{3;1} $}
  \put(105,-9) {\footnotesize $ X_{4;1} $}
  \put(162,126){\footnotesize $ X_{6;1} $}
  \put(-37,58) {\begin{turn}{23}\footnotesize $\tilde Y_{1;1} = Y_1 = \tilde Y_{2;3}$\end{turn}}
  \put(113,84) {\begin{turn}{60}\footnotesize $\tilde Y_{2;1} = Y_2 = \tilde Y_{5;3}$\end{turn}}
  \put(-32,16) {\begin{turn}{24}\footnotesize $\tilde Y_{2;2} = Y_3 = \tilde Y_{3;1}$\end{turn}}
  \put(80,58)  {\begin{turn}{-93}\footnotesize$\tilde Y_{3;2}$\end{turn}}
  \put(76,35)  {\begin{turn}{-93}\footnotesize$         = Y_4 \,{=}\, \tilde Y_{4;1}$\end{turn}}
  \put(146,32) {\begin{turn}{-22}\footnotesize$\tilde Y_{4;2} = Y_5 = \tilde Y_{5;2}$\end{turn}}
  \put(164,76) {\begin{turn}{-25}\footnotesize$\tilde Y_{5;1} = Y_6 = \tilde Y_{6;1}$\end{turn}}
  }
\end{picture}
  \caption{A labeling of the boundary circles of the components $(\surf_l,\mark_{\!l})$
  by objects $X_{l;i}$ and $\tilde Y_{l;j}$ as described in the text before \erf{eq:block2},
  for the cut surface that results from the marking shown in Figure \ref{fig2}.
  \label{fig3}}
\end{figure}
\end{center}

\noindent (4)\, \emph{Higher genus surfaces.}
\\[2pt]
Let now $\surfm$ be a marked surface of arbitrary genus $g$. Then we define the functor
$\mBl_\surfn$ by
  \be
  \mBl_\surfn(X_{1;1},...,X_{\ell;n_\ell})
  := \oint^{Y_1 \boxtimes \,\cdots\, \boxtimes Y_{|C|} \in \D^{\boxtimes |C|}} \!
  \bigotimes_{l=0}^\ell
  \mBl_{\surf_l,\mark_{\!l}}(X_{l;1},...,X_{l;n_l},\tilde Y_{l;1},...,\tilde Y_{l;m_l}) \,.
  \labl{eqdef:mBl-g}
Here the conventions concerning the objects $X_{l;i}$ and $\tilde Y_{l;j}$ are the same
as above, while when taking coends we now need to work explicitly with coends $\oint$ 
in categories of left exact functors. (Left-exactness, and thus representability, would 
no longer necessarily be preserved when taking coends in the category of all functors from
the appropriate Deligne power of \D\ to \Vect.)
The Fubini theorem in the form of \eqref{eq:Fubini2} allows us to rewrite 
$\mBl_\surfn$, similarly as in the case of spheres, as an iterated coend:
  $$
  \mBl_\surfn(X_{1;1},...,X_{\ell;n_\ell})
  = \oint^{Y_1 \in \D}\!\!\!\!\! \cdots \oint^{Y_{g} \in \D} \!\!\!
  \int^{Y_{g+1} \in \D}\!\!\!\!\! \cdots \!\int^{Y_{|C|} \in \D} \!
  \bigotimes_{l=0}^\ell
  \mBl_{\surf_l,\mark_{\!l}}(X_{l;1},...,X_{l;n_l},\tilde Y_{l;1},...,\tilde Y_{l;m_l}) \,.
  $$
Here the subset $\{c_1,c_2,...\,,c_g\} \,{\subset}\, C$ consisting of those cuts that
correspond to the objects $Y_k$ with $k \iN \{1,2,....\,,g\}$ needs to be selected in such
a way that the corresponding cut surface $\oline\surf{\{c_1,...c_g\}}$ has genus zero. 
Given this description of $\mBl_\surfn$, by recalling the genus-0 result \erf{defBl03} 
and invoking iteratively the relations \erf{deltaHom} and \erf{oint=HomK} 
one obtains for any connected marked surface \surfm\ a distinguished isomorphism
  \be
  \mBl_\surfn(-, ...\, ,-)
  \cong \HomD(\one,-\wee \oti \cdots \oti -\wee \oti K^{\otimes g})
  \labl{defBlgn}
of left exact functors, with $K \iN \D$ the coend \erf{K}.

\begin{rem}
As we are working with coends in functor categories, the prescription for higher genus
applies directly to surfaces for which each connected component has non-empty boundary.
But once $\mBl_\surfn$ is defined for all such surfaces, we can define $\mBl_\surfn$ for
a surface $\surf$ with empty boundary as the vector space $\mBl_{\surf',\mark'}(\one)$
with $\one$ the monoidal unit of \D\ and $\surf'$ obtained from $\surf$ by removing a disk.
\end{rem}

\medskip

For any extended surface $\surf$ there is a functor $U_\surf$ from the groupoid 
$\CWm(\surf)$ of fine markings on $\surf$ to the category $E/\!/\maps$ with a 
single object $\surf$ and with morphisms given by the mapping class group $\maps$
that forgets the structure of a marking (similar to the functor $U$ 
that will be introduced in Definition \ref{def:unmarking} below).
One can construct the conformal block functor $\Bl_\surf$ \erf{BlE} from the functors
$\mBl_\surfn$ by a right Kan extension along $U_\surf$. The so obtained
conformal block functors obey analogues of \erf{eq:Bl-sqcup} and \erf{defBlgn}.
We do not give any further details of the construction of $\Bl_\surf$ because 
they will not be needed in the sequel.

\begin{rem}
For any modular finite ribbon category \D\ the conformal blocks obtained by 
the prescription above provide equivalent representations of the mapping class groups
as those obtained in \cite{lyub11} by a different construction. The variant presented
above is tailored to our goal of determining consistent systems of correlators in the
sense of Definition \ref{def:Corr}. Apart from the precise treatment of boundary circles 
(as either incoming or outgoing), for the case that \D\ is a semisimple modular tensor
category it reduces to the construction in \cite{baKir}.
\end{rem}
 
The correlators we are looking for are elements of very specific conformal blocks spaces
$\Bl_\surf(X_1,...,X_n)$: those for which each of the arguments $X_i$ 
is one and the same object of \D, namely the bulk object $F$.
Nevertheless we had to discuss also block spaces with generic arguments $X_i$:
Compatibility with the sewing of surfaces is part of the consistency 
requirements to be imposed on correlators. As we will see in Section \ref{ss:pinned},
to formulate sewing we need certain structure morphisms of coends, and to get these
morphisms we need to consider blocks for arbitrary insertions.


\section{Consistency conditions for correlators}\label{sec3}

The purpose of this section is to give a concise definition of the notion of a
consistent system of bulk field correlators for local conformal field theories on 
oriented surfaces. In short, the correlators must be invariant under the action of
mapping class groups on conformal blocks, must be compatible with the sewing of surfaces,
and must obey a non-degeneracy condition.
To formalize these conditions, as compared to Section \ref{sec2} we change our perspective in
two respects: First, in Section \ref{sec:mBl} we described conformal blocks for (extended)
surfaces $\surf$ with arbitrary objects of \D\ associated to the boundary circles of $\surf$,
and thereby dealt with a system of functors \erf{BlE} of conformal blocks. In contrast, for 
the correlators we need to associate to every boundary circle of $\surf$ one and the 
same object (respectively its dual, in the case of an incoming boundary circle), the 
\emph{bulk object}. In case the category \D\ has a representation theoretic interpretation, 
the vector space underlying this object is the space of states that is related to 
bulk fields under a field-state correspondence.

Accordingly we now select 
one specific object $F$ of \D\ as a (candidate) bulk object. Thus for each surface $\surf$
we are now dealing with a \emph{vector space} of conformal blocks, endowed with an action of
the mapping class group $\maps$. The correlator for $\surf$ is a vector in this space; it is
required to be invariant under the $\maps$-action. 
Second, previously we treated one surface at a time, e.g.\ associated, in Section 
\ref{sec:fine}, a groupoid $\CWm(\surf)$ separately to each surface. In contrast, the 
sewing constraints require the system of correlators to be compatible with sewings 
that connect correlators on different surfaces. Accordingly we now study all surfaces 
together, and in particular treat the morphisms in all the groupoids $\CWm(\surf)$ as well 
as sewings of surfaces on the same footing. (Since sewing involves a sum over intermediate
states, and thus a coend, the considerations in Section 2 are necessary for this approach.)

A major step in this section will therefore be to construct, for a chosen object $F$ of \D, 
a symmetric monoidal functor $\BlF$, to be called the \emph{$F$-pinned block functor}, 
or just pinned block functor, from a suitable category of surfaces to \Vect. 
We will then see that the consistency conditions for correlators with $F$ as the
(candidate) bulk object can be neatly summarized as the requirement that these vectors
define a monoidal natural transformation, satisfying a simple non-degeneracy condition,
from a certain trivial functor to the pinned block functor $\BlF$. Our first task
will be to introduce the relevant categories of surfaces and marked surfaces.
This is initiated in Section \ref{ss:Surf} and completed in Section \ref{ss:cext}, where
we accommodate the fact that the mapping class groups act on conformal blocks 
only projectively.


\subsection{The categories of surfaces and of marked surfaces}\label{ss:Surf}

The two geometric categories of our interest have (extended) surfaces, 
respectively marked surfaces, as objects. Their morphisms are generated by two
types of special morphisms: automorphisms respectively moves of a given surface on 
the one hand, and sewings of surfaces, as described 
in Definition \ref{def:standardmarking}(iv), on the other.

Recall that a move of a marked surface is a morphism of the groupoid $\CWm(\surf)$,
and that it can be presented as a finite sequence of the elementary moves listed in 
Section \ref{sec:fine}, According to Definition \ref{def:maps}, the elements of the 
mapping class group $\maps$ preserve the orientation of each boundary circle
and thus map the subsets of incoming and outgoing boundary circles to themselves. 
$\CWm(\surf)$ is in fact also intimately related to a larger group 
$\mapsx\,{\supset}\,\maps$, defined analogously as $\maps$, but
without the restriction to preserve the orientations of boundary circles §\cite{baKir}:
 %
First, the set of morphisms of $\CWm(\surf)$ is invariant under the obvious action of
\mapsx. Moreover, for every $\phi \iN \mapsx$ there is a move
$\mathrm m \eq \mathrm m(\phi)$ in $\CWm(\surf)$ that maps a given fine marking $(C,\mark)$
of $\surf$ to $(\phi(C),\phi(\mark))$ while, conversely, any finite sequence 
$\mathrm m\colon \surfm\,\To(\surf',\mark')$ of elementary moves in $\CWm(\surf)$
not involving the F-move (which changes the number of cuts) uniquely determines an
element $\phi_{\mathrm m}$ of \mapsx\ such that the action of $\phi_{\mathrm m}$
on a fine marking of $\surf$ reproduces the effect of the move $\mathrm m$. Recall e.g.\
that the B-move affects the markings in the same way as a certain braiding diffeomorphism.
 
We call a morphism $\mathrm m$ of $\CWm(\surf)$ an \emph{admissible move} 
iff the associated element $\phi_{\mathrm m}$ of the group \mapsx\ is contained in
the subgroup \maps, i.e.\ is a mapping class in the sense of Definition \ref{def:maps}.

\begin{defi}
~\\[2pt]
(i)\, The \emph{category \Surf\ of surfaces} is the monoidal category having extended
surfaces $\surf$ as objects and whose morphisms are 
pairs $(\varphi,\mso)$ consisting of a mapping class $\varphi \iN \maps$ 
followed by a sewing $\mso\colon \surf \To \Sqcup\surf$ of surfaces.
\\[2pt]
(ii)\, The \emph{category \mSurf\ of marked surfaces} is the monoidal category having
marked surfaces $\surfm$ with fine marking as objects and whose morphisms are pairs 
$(\mathrm m,\mso)$ consisting of an admissible move of any of the groupoids $\CWm(\surf)$ 
for $\surf\iN\Surf$ followed by a sewing $\surfm \To \Sqcup\surfm$ of marked surfaces.
\end{defi}

In both categories the tensor product is given by disjoint union and $\emptyset$ is 
the monoidal unit; both \Surf\ and \mSurf\ are symmetric monoidal.
In the definition we have suppressed the description of the relations among the 
pairs that define the morphisms.
Besides the relations in the individual mapping class groups $\maps$, respectively
those among the admissible moves of the individual groupoids $\CWm(\surf)$, they consist of 
obvious compatibility relations 
that express the composition of a sewing with a mapping class 
(respectively, with an admissible move) as the composition of a
uniquely determined mapping class (respectively, admissible move)
of the unsewn surface with the sewing. Such
relations are discussed in detail in \cite{haLS}; 
we refrain from writing any explicit formulas (compare also \Cite{Rem.\,5.6.4}{BAki}).

\begin{rem}
In \Cite{Sect.\,5.6}{BAki} the Teichm\"uller tower of mapping class groups is studied. 
It has the same objects as \Surf, but only mapping classes are taken as morphisms, while
sewings are regarded as an additional structure on the category. For the purpose of describing
conformal field theory correlators it is very convenient to take, as in \cite{lyub11}, also
sewings as morphisms. These are non-invertible; note that we do not introduce morphisms 
for the operation of cutting surfaces, which is inverse to the sewing operation.
\end{rem}

Note that sewing is a local construction. As a 
consequence, we can restrict our attention to suitable \emph{elementary} sewing morphisms 
involving only specific types of surfaces. For the sake of concreteness in explicit formulas 
it is, however, still convenient to treat two kinds of situations separately, namely
(in the case of \Surf, and analogously for \mSurf) the following: Either a sewing morphism
  \be
  \mss12^{} : \quad \gpq{g_1}{p_1}{q_1} \,{\sqcup}\,\gpq{g_2}{p_2}{q_2} \to
  \gpq{g_1+g_2}{p_1+p_2-1}{q_1+q_2-1} 
  \labl{elemsurf1}
from the disjoint union of two connected surfaces of genus $g_1$ and $g_2$ having
$n_1 \eq p_1\pl q_1$ and $n_2 \eq p_2\pl q_2$ holes, respectively, to a connected surface
of genus $g_1 \pl g_2$ with $n_1 \pl n_2 \mi 2$ holes; or else, a sewing morphism
  \be
  \ms^{} : \quad \gpq gpq \to \gpq{g+1}{p-1}{q-1}
  \labl{elemsurf2}
from a connected surface of genus $g$ with $n$ holes to a connected surface of
genus $g \pl 1$ with $n \mi 2$ holes.

While apart from the sewings, the morphisms of \mSurf\ and \Surf\ are quite different,
owing to the relation between morphisms of $\CWm(\surf)$ and the group \mapsx\ there is
nevertheless a natural functor from \mSurf\ to \Surf\ that is a kind of forgetful functor.
Its action on objects and on sewings is the obvious one, namely forgetting the marking.

\begin{defi}\label{def:unmarking}
The \emph{unmarking functor} $\U\colon \mSurf\To\Surf$ is the symmetric monoidal functor
that is uniquely determined by the following prescription.
\\[1pt]
(i)\, On objects and on sewings, $\U$ forgets the marking, i.e.\ $\U\surfm \,{:=}\, \surf$ and
  $$
  \U\big( \surfm\,{\to}\,\Sqcup\surfm \big) := (\surf\,{\to}\,\Sqcup\surf) \,.
  $$
(ii) The F-move $\mF\colon \surfm\,{\to}\,\surfmp$ is mapped to the identity, 
$\U(\mF) \,{:=}\, \id_\surf$.
\\[2pt]
(iii)\, The other elementary moves $\mathrm m$ of $\CWm(\surf)$ -- the Z-move, 
B-move, A-move and S-move -- are mapped to the mapping class $\phi_{\mathrm m}$ 
that reproduces the effect of $\mathrm m$ on the marking of $\surf$.
\end{defi}

We would now like to fix an object $F$ of \D\ and construct the $F$-pinned block functor 
that provides the conformal block spaces for bulk fields. We will first obtain it as 
a monoidal functor $\mBlF$ from marked surfaces to \Vect\ and then upon Kan extension along
the unmarking functor get a monoidal functor $\BlF$ from (extended) surfaces to \Vect.
Defining $\mBlF$ on an object \surfm\ of \mSurf\ is easy;
we just apply the functor $\mBl_\surfn$ introduced in \erf{mBlE} to 
the object $F^{\boxtimes |\piods|} \iN \D^{\boxtimes |\piods|}$:

\begin{defi}\label{def:mBlF1}
To any object $\surfm$ of \mSurf, $\mBlF$ assigns the \findim\ vector space
  \be
  \mBlF\surfm := \mBl_\surfn(F,F,...\,,F)
  \ee
with the appropriate number $|\piods|$ of arguments of $\mBl_\surfn$.
\end{defi}


To define $\mBlF$ also on morphisms is somewhat less straightforward:
there is an obstruction, known as the framing anomaly, which results from the fact that
the action of $\maps$ on the space $\mBlF\surfm$ is projective. We deal with this
obstruction by adequately extending the categories \mSurf\ and \Surf.


\subsection{Central extensions of categories of surfaces}\label{ss:cext}

As a first step of defining the pinned block functor $\mBlF$ on morphisms we
consider moves of marked surfaces. This will in particular demonstrate the need
to work with suitable extensions of the categories \mSurf\ and \Surf. 
Besides these extensions our construction will involve 
the dinatural structure morphisms for the coends \erf{deltaHom} 
and \erf{oint=HomK}, as well as suitable specific families of bijective linear maps 
which correspond to the elementary moves of marked surfaces.

We start by introducing the latter families of linear maps. Let \D\ be a modular 
finite ribbon category. Our prescription is entirely based on structural data of \D,
including those which are captured by the Hopf algebra $K \iN \D$:

\begin{defi}\label{def:Z...S-iso} 
Denote the braiding of \D\ by $c$, the evaluation and coevaluation for the right duality
of \D\ by $d$ and $b$, respectively, and the pivotal structure of \D\ by $\pi$
(i.e.\ $\pi_X\colon X\To X^{\vee\vee}$).
\\[4pt]
(i)\, For $X,Y\iN\D$, the \emph{Z-isomorphism} $\mZ_{X,Y}\colon \HomD(\one,X\oti Y)
\To \HomD(\one,Y\oti X)$ is the linear map given by
  \be
  \mZ_{X,Y}(f) := (d_X \oti \id_{Y\otimes X}) 
  \circ (\idXv \oti f \oti \piv_X^{-1}) \circ b_{X^\vee_{}} \,.
  \labl{def-Z-iso}
(ii)\, For $X,Y,U\iN\D$, the \emph{B-isomorphism} 
$\mB_{X,Y,U}\colon \HomD(\one,X\oti Y\oti U) \To \HomD(\one,Y\oti X \oti U)$
is the linear map given by
  \be
  \mB_{X,Y,U}(f) := (c^{}_{X,Y} \otimes \id_U^{}) \circ f \,.
  \labl{def-B-iso}
(iii)\, For $U,V\iN\D$, the \emph{F-isomorphism} 
  $$  
  \mF_{U,V} :\quad \int^{X\in\D}\!\! \HomD(\one, X\oti V) \otik \HomD(\one,U\oti X^\vee) 
  \xrightarrow{~~} \HomD(\one,U\oti V)
  $$ 
is the linear map that on morphisms $f\oti g$ in $\HomD(\one,X\oti V) \otik \HomD(\one,U\oti X^\vee)$
acts as 
  \be
  f\oti g \,\longmapsto\, (\id_U \oti d_X \oti \id_V) \circ (g \oti f) \,.
  \labl{def-F-isoUV}
(iv)\,
For $U,U',V,V'\iN\D$, the \emph{A-isomorphism} $\mA_{U,U',V,V'}$ is the composition
  \be
  ~\hsp{-.8} \bearl\dsty
  \int^{X\in\D}\!\! \HomD(\one, U\oti U'\oti X) \otik \HomD(\one,V\oti V'\oti X^\vee) 
  \Nxl1 \dsty \hsp{2.8}
  \xrightarrow{~=~}~ \HomD(\one,U\oti U'\oti V\oti V')
  ~\xrightarrow{~\mZ_{U,U'\otimes V\otimes V'}~}~ \HomD(\one,U'\oti V\oti V'\oti U)
  \Nxl2 \dsty \hsp{9.4} 
  \xrightarrow{~=~} \int^{Y\in\D}\!\!
  \HomD(\one, U'\oti V\oti Y) \otik \HomD(\one,V'\oti U\oti Y^\vee) \,,
  \eear
  \labl{def-A-iso}
where the equalities indicate the identifications of coends that result when applying
formula \erf{deltaHom} with $G$ an appropriate Hom functor.\,%
 \footnote{~It is worth noting that the map \erf{def-F-isoUV} is nothing but
 the dinatural structure morphism for the coend in question, so that
 $\mF_{U,V}$ is the identification of the vector space $\HomD(\one,U\oti V)$ as the
 coend $\int^{X\in\D} \HomD(\one, X\oti V) \,{\otimes_\ko}\, \HomD(\one,U\oti X^\vee)$.
 Accordingly it is appropriate to write the first and third maps in the A-iso\-mor\-phism 
 \erf{def-A-iso} as equalities.}
\\[3pt]
(v)\,
Denote by $\eps_K$ the counit and by $\Lambda_K$ a non-zero two-sided integral of the 
Hopf algebra
$K$, and recall that $\iK$ denotes the dinatural family of $K$ as a coend.
Define $\mathcal Q_K \iN \EndD(K\oti K)$ by
  $$ 
  \mathcal Q_K \circ (\iK_X \oti \iK_Y) := (\iK_X \oti \iK_Y) \circ
  \big( \id_X^{} \oti (c_{Y,X^\vee_{\phantom,}}^{} \cir c_{X^\vee_{\phantom,}\!,Y}^{})
  \oti \id_{Y^\vee_{\phantom,}} \big) 
  $$ 
and set \cite{lyub8,lyub6}
   \be
   S^K := (\eps_K \oti \id_K) \circ \mathcal Q_K \circ (\id_K \oti \Lambda_K) 
   \,\in \EndD(K) \,.
   \labl{SK}
Then for $U\iN\D$, the \emph{S-isomorphism} $\mS_U$ is the linear endomorphism of
$\oint^{X\in\D}\! \HomD(U,X\oti X^\vee)
        $\linebreak[0]$
{=}\, \HomD(U,K)$ that acts as post-composition by the isomorphism $S^K$.
\end{defi}

\begin{rem} \label{rem:g=0}
Note that the parts (i)\,--\,(iv) of Definition \ref{def:Z...S-iso} only
require \D\ to be a finite ribbon category. Part (v) uses a two-sided integral $\Lambda_K$;
such an integral is guaranteed to exist if \D\ is even modular.
Similarly, in the analysis below modularity will be essential when dealing with
the relations {\rm (W11)\,--\,(W13)}, while it is not needed as long as only the
relations {\rm (W1)\,--\,(W10)} are concerned.
\end{rem}

It is known (see \Cite{Sect.\,8.8}{lyub11} and \Cite{Thm\,2.1.9}{lyub6}) that the
normalization of the integral $\Lambda_K$ can be chosen (uniquely up to sign) in such a
way that the square of the endomorphism $S^K$ in \erf{SK} equals the inverse of the
antipode $\apo_K$ of $K$,
  \be
  \big(S^K{\big)}^2_{} = \apo_K^{-1} .
  \labl{SS=apoi}
In the sequel we take $\Lambda_K$ to be normalized in this way.

\medskip

We can now define $\mBlF$ on generating (and thus not necessarily admissible) moves.
Let us first recall the notation $\eps \iN \{\pm1\}$ indicating whether
a boundary circle is outgoing ($\eps \eq 1$) or incoming ($\eps \eq {-}1$),
and the corresponding notation $X\wee$ standing for $X \iN \D$ if $\eps \eq 1$ and
for $X^\vee$ if $\eps \eq {-}1$. We supplement these conventions by setting
  $$
  \jef := \left\{ \bearll \id_{F^\vee_{}\otimes F}^{} \,\in \EndD(F^\vee_{}{\otimes}\, F)
  & {\rm for}~ \eps \eq 1 \,, \Nxl2 
  \pi_F^{} \oti \idFv \,\in \HomD(F\oti F^\vee_{}, F^{\vee\vee}_{}{\otimes}\, F^\vee_{})
  & {\rm for}~ \eps \eq {-}1 \,, \eear \right.
  $$
with $\pi$ the pivotal structure of \D.

\begin{defi}\label{def:mBlF2}
On generating moves $\mathrm m\colon (\surf,\mark) \To (\surf,\mark')$ of 
the groupoid $\CWm(\surf)$ the assignment $\mBlF$ is defined as follows: To each 
elementary move assign the specific linear isomorphism $\mBlF\surfm \To \mBlF\surfmp$
from Definition \ref{def:Z...S-iso} that bears the same name as the move
and for which all objects of $\D$ involved are given by either $F$ or $F^\vee$. 
Explicitly, if all boundary circles of $\surf$ are outgoing, then
  $$
  \mBlF(\mB) := \mB_{F,F,F} \,,\qquad
  \mBlF(\mA) := \mA_{F,F,F,F} \qquad {\rm and} \qquad
  \mBlF(\mS) := \mS_{F} \,,
  $$
as well as
  $$
  \mBlF(\mZ) := \mZ_{F,F\otimes F} \qquad {\rm and} \qquad
  \mBlF(\mF) := \mF_{F,F\otimes F} 
  $$
in case $\surf$ has three holes, and similarly if $\surf$ has less than three holes.
If any of the boundary circles of $\surf$ are incoming, the appropriate
occurrences of $F$ are to be replaced by $F^\vee$.
\end{defi}

For an arbitrary move $\mathrm m$ one might wish to define $\mBlF(\mathrm m)$ to be the 
composition of isomorphisms from Definition \ref{def:mBlF2} according to the expression 
of $\mathrm m$ as a sequence of elementary moves. However, this works directly only
if those isomorphisms respect the relations among elementary moves. We first note

\begin{lem}\label{lem:W1-11-13}
Applying $\mBlF$ to any of the relations {\rm (W1)\,--\,(W11)} and {\rm (W13)}
listed in Section {\rm \ref{sec:fine}} yields an equality of linear isomorphisms.
\end{lem}

\begin{proof}
For (W1) and (W2) the claim follows directly from the definitions, while for (W3) and 
(W4) one has to make use of the defining properties of the pivotal structure.
As another example, the proof for (W7) follows by rewriting the Z-isomorphism as
  \be
  \mZ_{X,Y}^{}(f) = (\id_Y^{} \oti \theta_X^{}) \circ c_{X,Y}^{} \circ f
  = (\id_Y^{} \oti \theta_X^{-1}) \circ c_{Y,X}^{-1} \circ f \,,
  \labl{Z-theta}
which uses the relation between the pivotal structure and the twist and braiding
of the ribbon category \D.
Let us give some details for the case of (W11), i.e.\ the genus-1 relation 
$\mB_{c,\alpha} \cir \mZ \eq \mS^2$, with $\alpha$ the boundary of a one-holed torus
that is endowed with a cut system consisting of a single cut $c$.
To prove the assertion we have to deal with morphisms involving the Hopf algebra 
$K$, which by the coend property of $K$ amount to \emph{families} of morphisms.
Specifically, one finds that the family realizing the move $\mB_{c,\alpha} \cir \mZ$ 
consists of the linear maps
  $$
  f \, \longmapsto\,
  (\id_U^{} \oti \idXv \oti \pi_X^{-1}) \circ \mZ_{X,U\otimes X^\vee}^{}
  \circ (B_{U,X}^{}(f) \oti \idXv) \,\in \HomD(\one,U\oti X^\vee \oti X^{\vee\vee}) \,
  $$
with $f \iN \HomD(\one,U\oti X \oti X^\vee)$, for all $X\iN\D$. Using \erf{Z-theta} one 
can see that these give the endomorphism $f \,{\mapsto}\, (\id_U^{} \oti \apo_K^{-1}) \cir f$
of $\HomD(\one,U\oti K)$, with $\apo_K$ the antipode of the Hopf algebra $K$. Hence 
invoking the equality \erf{SS=apoi} (and thereby adopting the corresponding choice of
normalization of $\Lambda_K$),
indeed post-composition with $\id_U^{} \oti \apo_K^{-1}$ implements the square 
of the S-iso\-morphism on $\HomD(\one,U\oti K)$, as required to realize the relation (W11). 
Finally we mention that in the proof for (W13), a crucial additional ingredient is 
to invoke an isomorphism of the form $\oint^{Y\in\D}\! \int^{X\in\D} G(U;X,X,Y,Y) 
\,{\cong}\, \oint^{X\in\D}\! \int^{Y\in\D} G(U;X,X,Y,Y)$, which exists and is 
uniquely determined as a consequence of the Fubini theorem \eqref{eq:Fubini2}.
\end{proof}

It remains to examine the relation (W12), i.e.\ the modular group relation
$ (\mS \cir \mT_c)^3_{} \eq \mS^2_{} $, where $\mT_c$ is the Dehn move around the 
single cut $c$ of a one-holed torus.
Again we deal with morphisms involving the coend $K$ and thus work with families of
morphisms. We find that the family realizing the Dehn move $\mT_c$ is
  $$
  f \, \longmapsto\,
  \mZ^{-1}_{U,X\otimes X^\vee} \circ \big( \mB^{-1}_{X,X^\vee\otimes U} \cir
  Z^{}_{X,X^\vee\otimes U} \big) \circ \mZ^{}_{U,X\otimes X^\vee} \circ f
  $$
with $f \iN \HomD(\one,U\oti X \oti X^\vee)$, for all $X\iN\D$. Upon again invoking
\erf{Z-theta}, this amounts to the linear endomorphism of $\HomD(\one,U\oti K)$ that is 
given by post-composition with $\id_U^{} \oti T^K$, where $T^K \iN \EndD(K)$ is determined by 
  $$
  T^K \circ \iK_X = \iK_X \circ (\theta_X^{} \oti \idXv)
  $$
with $\theta$ the twist of \D. Now by 
\Cite{Thm.\,2.1.9}{lyub6} the morphisms $S^K$ and $T^K$ obey 
the modular group relation up to a scalar factor,
  \be
  {(S^K\cir T^K)}^3 = \zeta\, {(S^K)}^2
  \labl{eq:zeta}
with $\zeta \,{:=}\, \eps_K \cir T^K \cir \Lambda_K \iN \ko^\times$. The number $\zeta$,
which via its dependence on the integral $\Lambda_K$ is determined up to sign, is called
the \emph{central charge} or \emph{framing anomaly}. Its presence in \erf{eq:zeta}
obstructs a linear realization of the morphisms of \mSurf. In terms of the category \Surf\ 
we then get projective rather than genuine representations of mapping class groups.

\begin{rem}\label{rem:redef1}
At genus one the projective action of the mapping class group can actually be reduced to a
genuine
linear action upon redefining the action of the Dehn twist $\mT_c$ \cite{atiy18}. However,
this would not lead to genuine actions at higher genus, compare Remark 3.1.9 of \cite{BAki}.
\end{rem}

We can trade these projective realizations for linear ones by considering suitable
central extensions of categories of surfaces \Cite{\S4}{sega16}. 
Analogously as in \Cite{Sect.\,7}{lyub11} (compare also \Cite{Sect.\,5.7}{BAki}), 
in terms of generators and relations for morphisms of marked surfaces the required
central extension is implemented as follows. First, introduce for each connected surface 
$\surf$ with marking $\mark$ a new invertible generator $\mC_\surfm$. Second, require
that these generators are compatible with (not necessarily admissible) moves $\mathrm m$,
in the sense that
  \be
  \mathrm m \circ \mC_\surfm = \mC_{\surfmp} \circ \mathrm m 
  \labl{eq:mC=Cm}
for any move $\mathrm m\colon \surfm \To \surfmp$, and keep all relations 
(W1)\,--\,(W11) and (W13) among the generating moves, while replacing the modular group 
relation (W12) by
 
\bigskip

  $(\mathrm W12)\cc : $
\\[-29pt]
  $$
  (\mS \circ \mT_c)^3_{} = \mC \circ \mS^2_{} 
  $$

\smallskip \noindent
with $\mC \eq \mC_{(T,\mark)}$ the new generating morphism for a one-holed torus $T$
with marking $\mark$ as described for (W12) in Section \ref{sec:fine}.
And third, impose $\mC_\emptyset \eq \id_\emptyset$ as well as the relation
  \be
  \mss12^{} \circ \big( \mC_{(\surf_1,\mark_1)}^{\,\,k_1} \,{\sqcup}\,
  \mC_{(\surf_2,\mark_2)}^{\,\,k_2} \big) =  \mC_{\surfm}^{\,k_1+k_2}  \circ \mss12^{}
  \labl{eq:lyub11-7.2}
for any $k_1,k_2 \iN \zet$ and any sewing $\mss12 \colon (\surf_1,\mark_{\!1})
\,{\sqcup}\, (\surf_2,\mark_{\!2}) \To \surfm $ of the type \erf{elemsurf1}
among connected surfaces.
Thus we introduce, similarly as in \Cite{Def.\,7.2}{lyub11}:

\begin{defi}\label{def:mSurfC} 
The \emph{central extension \mSurfC\ of the category \mSurf\ of marked  surfaces} is 
the category with the same objects as \mSurf\ and with morphisms generated by the
morphisms of \mSurf\ together with the morphisms $\mC_\surfm$ for all connected marked
surfaces $\surfm$, subject to the following relations:
those obtained from the relations among morphisms of \mSurf\ when replacing
$(\mathrm W12)$ by $(\mathrm W12)\cc$; the relations \erf{eq:mC=Cm} for
all admissible moves $\mathrm m$; $\mC_\emptyset \eq \id_\emptyset$; and \erf{eq:lyub11-7.2} 
for any integers $k_1,k_2$ and any sewing of the type \erf{elemsurf1}.
\end{defi}
  
To see how to extend the definition of $\mBlF$ to the morphisms $\mC_{\surfm}$, recall that 
$\mBlF$ maps the moves $\mS$ and $\mT_c$ to post-composition by endomorphisms of the 
Hopf algebra $K$.
This must then likewise apply to the morphism $\mC$ in $(\mathrm W12)\cc$. To ensure 
compatibility
with \erf{eq:zeta} we set $\mBlF(\mC) \,{:=}\, ( \id_{F^\Wee_{}} \oti \mC^K {)}_*$ with 
  $$
  \mC^K := \zeta\, \id_K \,,
  $$
where $\zeta \iN \ko^\times$ is the scalar appearing in \erf{eq:zeta}. To also account for
the commutation relations \erf{eq:lyub11-7.2}, we generalize this prescription to arbitrary 
marked surfaces \surfm, to which $\mBlF$ assigns the space $\HomD(\one,X \oti K^{\otimes g}_{})$
(with $X$ an appropriate tensor product of factors $F$ and $F^\vee$), by
  \be
  \mBlF(\mC_{\surfm}) := \big( \id_X \otimes (\mC^K)^{\otimes g}_{} {\big)}_* \,.
  \labl{eq:mBlF-C}
Note that this way effectively all central extensions are reduced to the choice of 
the single number $\zeta$.

Next recall the unmarking functor $\U\colon \mSurf \To \Surf$ introduced in Definition
\ref{def:unmarking}.
Analogously as for marked surfaces we can centrally extend \Surf\ to a category
\SurfC\ that has a new central automorphism $\mC_\surf$ for every connected surface $\surf$,
and because of the relations \erf{eq:mC=Cm} we can immediately extend $\U$ to a functor from
\mSurfC\ to \SurfC\ by the prescription $\mC_{\surfm} \,{\mapsto}\, \mC_\surf$, i.e.\
by simply forgetting the marking. This allows us to give

\begin{defi}\label{def:SurfC} 
Let $p$ be the natural projection of \mSurfC\ to \mSurf\ that exists by Definition 
\ref{def:mSurfC}. The \emph{central extension \SurfC\ of the category \Surf\ of surfaces} is 
the central extension of \Surf\ that has the same objects as \Surf\ and whose morphisms 
are determined by commutativity of the diagram 
  $$
  \begin{tikzcd}
  \mSurfC \ar{r}{p} \ar{d}[swap] {\UC} & \mSurf \ar{d}{\U}
  \\
  \SurfC \ar{r}{} & \Surf 
  \end{tikzcd}
  $$
where $\UC$ is the functor that acts as $\U$ on all objects and on all morphisms of \mSurfC\
that are morphisms of \mSurf, and by by mapping central elements to central elements.
\end{defi}

The morphisms of \SurfC\ are thus generated by sewings and by elements of
central extensions \mapsC\ of the mapping class groups \maps. These groups \mapsC\ 
have e.g.\ been described in \cite{maRob,gervS2}. 

\begin{rem}
Recall from Remark \ref{rem:redef1} that by a suitable redefinition of the action of the 
Dehn twist $\mT_c$ one can achieve genuine actions of the mapping class groups at genus 
one. The conditions \erf{eq:lyub11-7.2}, while providing relations between the central terms 
at any genus $g \,{>}\, 1$ and the one at genus 1, are too weak to imply such a simplification
also at higher genus.
\end{rem}


\subsection{The pinned block functor}\label{ss:pinned}

We finally turn our attention to sewings.
Recall that the construction in Section \ref{sec:mBl} gives the conformal blocks as 
coends. We define the functor $\mBlF$ on sewings with the help of the corresponding
dinatural structure morphisms of the coends \erf{eq:block2}. For the sake of giving
explicit formulas we invoke the Fubini theorems \eqref{eq:Fubini} and \eqref{eq:Fubini2}.
to write these coends as iterated coends, again
as in Section \ref{sec:mBl}. This requires to treat the elementary sewings of the form
\erf{elemsurf1} and \erf{elemsurf2} separately, even though sewing is a local operation:
in the case of \erf{elemsurf1} we deal with a coend to which the formula \erf{deltaHom}
applies, while for a sewing as in \erf{elemsurf2} the result \erf{oint=HomK} is relevant.
This leads us to

\begin{defi}\label{def:mBlF3}
On elementary sewings $\mBlF$ acts as follows:
\\[3pt]
To a sewing $\mss12 \colon (\surf_1,\mark_{\!1}) \,{\sqcup}\, (\surf_2,\mark_{\!2})
\To \surfm \eq (\surf_1,\mark_{\!1}) \,{\cup_{\beta,\gamma}}\, (\surf_2,\mark_{\!2})$
of the type \erf{elemsurf1}, for which we have
  $$
  \bearll
  \mBlF(\surf_1,\mark_{\!1})\,{\sqcup}\, (\surf_2,\mark_{\!2})) \!\!&
  = \mBlF(\surf_1,\mark_{\!1}) \otik \mBlF(\surf_2,\mark_{\!2})
  \Nxl3 &
  = \HomD(\one,U \oti F^{\Wee} \oti X) \otik \HomD(\one,V \oti F^{-\Wee} \oti Y)
  \eear
  $$
with objects $U,V,X,Y \iN \D$ which are appropriate tensor products of $F$, $F^\vee$ and $K$,
and with appropriate $\eps \iN \{\pm1\}$, assign the linear map $\mBlF(\mss12)$ that maps 
a linear map $f_1 \oti f_2$ in $\mBlF(\surf_1,\mark_{\!1}) \otik \mBlF(\surf_2,\mark_{\!2})$ to
  \be
  \bearl
  \mBlF(\mss12)\,(f_1{\otimes}f_2)
  \Nxl2 \hsp4
  := \big[\, \id^{}_{U} \otimes ( d_{F^{-\Wee}_{}} \cir \jef )
  \otimes \id^{}_{Y \otimes V \otimes X} \big]
  \circ \big[\, \id^{}_{U \otimes F\wee_{}} \otimes \mZ^{}_{V.F^{-\Wee}_{}\otimes Y}(f_2)
  \otimes \id^{}_{X} \big] \circ f_1 \,.
  \eear
  \labl{eq:sewmap1}
To a sewing $\ms \colon (\surf,\mark) \To (\surf',\mark') \eq
\Sqcup_{\beta,\gamma} (\surf,\mark)$ of the type \erf{elemsurf2}, for which
$\mBlF\surfm \eq \HomD
    $\linebreak[0]$
(\one,U\oti F^{\Wee} \oti V \oti F^{-\Wee} \oti W)$ with appropriate $U,V,W \iN \D$
and $\eps \iN \{\pm1\}$, assign the linear map $\mBlF(\ms)$ that acts as
  \be
  \mBlF(\ms)\,(f)
  := \big[\, \id^{}_{U \otimes V} \otimes ( \iKFe \cir \jef ) \otimes \id^{}_W \big] \circ
  \big[\, \id^{}_U \otimes c^{}_{F^{\Wee}_{},V} \otimes \id^{}_{F^{-\Wee}_{} \otimes W} \big]
  \circ f 
  \labl{eq:sewmap2}
on $f \iN \mBlF\surfm$.
\end{defi}

Note that the linear map $\phi \eq\mBlF(\mss12)(f_1{\otimes}f_2)$ 
is by definition an element of the space $\HomD(\one,U\oti Y\oti V\oti X)$. This
morphism space is isomorphic
to the space of blocks for the marked surface \surfm, and is equal to the space 
$\mBlF(\surf,\mark')$ of blocks for some marking $\mark'$ on $\surf$. 
In the sequel we tacitly identify $\phi$ with $\psi \cir \phi \iN \mBlF\surfm$, where
$\psi\colon \mBlF(\surf,\mark') \To \mBlF\surfm$ is the distinguished isomorphism 
assigned by Definition \ref{def:mBlF2} to the move that transforms $\mark'$ to $\mark$.
Also recall that there are Fubini theorems which provide
unique isomorphisms between different realizations of $\mBlF\surfm$ as
morphism spaces of \D. By adequately composing with such isomorphisms as well as 
suitable Z-isomorphisms one can restrict to the case $V \eq \one$ in the
definition of $\mBlF(\ms)$.

We are now ready to state the following result, which may be seen as a counterpart of 
Theorem 8.1 of \cite{lyub11} in the framework of finely marked surfaces used here:

\begin{prop}\label{prop:mBlFfunctor}
Define $\mBlF\colon \mSurfC \To \Vect$ as follows.
\\[2pt]
{\rm (i)}\, $\mBlF$ acts on objects of the category \mSurfC\ as prescribed in 
Definition {\rm \ref{def:mBlF1}}.
\\[1pt]
{\rm (ii)}\, $\mBlF$ acts on moves as prescribed in Definition {\rm \ref{def:mBlF2}},
on elementary sewings as in Definition {\rm \ref{def:mBlF3}},
and on the morphisms $\mC_{\surfm}$ as in formula {\rm \erf{eq:mBlF-C}}.
\\[1pt]
{\rm (iii)}\, To any presentation of a morphism of \mSurfC\ as a word in the morphisms 
from {\rm (ii)}, $\mBlF$ assigns the corresponding composition of linear maps.
\\[3pt]
Then $\mBlF$ constitutes a symmetric monoidal functor from \mSurfC\ to \Vect.
\end{prop}

\begin{proof}
To any presentation of a morphism $\mathrm m\colon (\surf_1,\mark_1) \To (\surf_2,\mark_2)$ 
of \mSurfC, $\mBlF$ assigns a linear map $\mBlF(\surf_1,\mark_1) \To \mBlF(\surf_2,\mark_2)$.
To prove that $\mBlF$ is a well-defined, we must show that this linear map in fact depends 
only on $\mathrm m$, but not on the chosen presentation of $\mathrm m$ or, equivalently,
that the prescriptions (ii) for specific types of morphisms respect all relations.
Then $\mBlF$ is in particular also compatible with the composition of morphisms.
\\[2pt]
Compatibility with all relations among (not necessarily admissible) moves except those
involving the relation $(\mathrm W12)\cc$ holds by Lemma \ref{lem:W1-11-13}.
Compatibility with $(\mathrm W12)\cc$ holds as well: the linear maps obtained by
applying $\mBlF$ to the left and right hand sides of $(\mathrm W12)\cc$
are equal owing to \erf{eq:zeta} and \erf{eq:mBlF-C}.
Compatibility with all relations among sewings is guaranteed by Fubini theorems.
Finally, the fact that cutting and sewing are inverse operations allows one to reduce any 
relation between moves and an elementary sewing to a relation among moves of the sewn surface.
\\[2pt]
The symmetric monoidal structure is implied by \erf{eq:Bl-sqcup}: we have 
$\mBlF(\emptyset) \,{=}\, \ko$ as well as
  $$
  \mBlF\big( (\surf_1,\mark_1) \sqcup (\surf_2,\mark_2)\big)
  = \mBlF(\surf_1,\mark_1) \otik \mBlF(\surf_2,\mark_2) \,,
  $$
and analogously for morphisms. 
\end{proof}


Now recall that our goal is to study bulk correlators for a conformal field theory.
These are assigned to (extended) surfaces, rather than to marked surfaces. Accordingly
we are really looking for a symmetric monoidal functor $\BlF\colon \SurfC \To \Vect$, rather 
than the functor $\mBlF$ constructed above for marked surfaces. 
But we can obtain $\BlF$ easily from $\mBlF$, namely as a Kan extension:

\begin{prop}
{\rm (i)}\, The right Kan extension
  \be
  \begin{tikzcd}
  \mSurfC \ar{rr}{\mBlF}[name=endNatTraf,below]{} \ar{d}[swap] {\UC} &~& \Vect
  \\
  \SurfC \ar[dashed]{urr}[swap,xshift=-6pt] {\BlF}
  \ar[Rightarrow,xshift=-5pt,to path=-- (endNatTraf) \tikztonodes]{}{}
  \end{tikzcd}
  \labl{eq:theKan}
of $\mBlF$ along the unmarking functor $\UC$ exists. 
\\[3pt]
{\rm (ii)}\, The so defined functor $\BlF\colon \SurfC \To \Vect$ has a natural 
symmetric monoidal structure.
\end{prop}

\begin{proof} (i)\,
For each $\surf \iN \SurfC$ consider the natural projection $Q_\surf\colon (f,\surfpm)
\,{\mapsto}\, \surfpm$ from the comma category $\comcat\surf$ to \mSurfC.
The limit of the functor $\mBlF \cir Q_\surf\colon \comcat\surf \To \Vect$ not only
exists (as for any functor to \Vect), but it can also be realized concretely. Indeed,
since up to unique isomorphism $\mBlF\surfm$ only depends on the underlying surface 
$\surf$ of \surfm, as a vector space the limit can be realized, up to distinguished
isomorphism, simply by $\mBlF(\surf,\marke)$ for any reference choice of fine marking 
$\marke$ on $\surf$. 
 %
It follows \Cite{Thm.\,X.3.1}{MAcl} that the right 
Kan extension of $\mBlF$ along $\UC$ exists and is realized by taking the limit 
of $\mBlF \cir Q_\surf$ at each object and each morphism of \SurfC.
\\[2pt]
(ii)\, We are free to choose the auxiliary markings $\marke$ arbitrarily for all connected
surfaces $\surf$ and to set $\mark_{\!\surf\sqcup\surf'} \,{:=}\, \marke \,{\sqcup}\, \markep$.
Then we have
  $$
  \bearll
  \BlF(\surf{\sqcup}\surf') \!\!\!& = \mBlF(\surf{\sqcup}\surf',\mark_{\!\surf\sqcup\surf'})
  = \mBlF((\surf,\marke) \,{\sqcup}\, (\surf',\markep)) 
  \Nxl3&
  = \mBlF(\surf,\marke) \otik \mBlF(\surf',\markep) = \BlF(\surf) \otik \BlF(\surf') \,.
  \eear
  $$
Thus indeed $\BlF$ is (strict) symmetric monoidal.
\\
(With the same prescription of auxiliary data the natural transformation
$\psi\colon \BlF\cir\UC \,{\Rightarrow}\, \mBlF$ that is part of the Kan extension is 
monoidal as well: 
By construction the linear map $\psi_\surfm^{}$ is nothing but the distinguished isomorphism 
$\mBlF(\surf,\marke) \,{\xrightarrow{~\cong~}}\, \mBlF\surfm$. It thus follows directly that 
$\psi_{(\surf\sqcup\surf',\mark\sqcup\mark')}^{}
\eq \psi_\surfm^{} \otik \psi_{(\surf',\mark')}^{}$.)
\end{proof}



\subsection{Systems of correlators as natural transformations}


Having at hand the functor $\BlF$ it is easy to formulate the consistency conditions that 
have to be obeyed by a system of correlators: The correlator $\vsms$ for a surface $\surf$ 
is an element of the space $\BlF(\surf)$; it is required that $\vsms$ is invariant under 
the centrally extended mapping class group of $\surf$,
  $$
  \BlF(\phi)\,\big(\vsms\big) = \vsms \qquad{\rm for~any}\quad \phi \iN\mapsC \,,
  $$
and that correlators for surfaces that are related by sewing are mapped to each other,
  $$
  \BlF(\mathsf s)\,\big(\vsms\big) = \vsmsp \qquad{\rm for~any~sewing}\quad
  \mathsf s\colon \surf\To\surf' \,.
  $$
In addition, to exclude degenerate solutions, we require that the endomorphism of $F$
that is provided by the correlator for a
cylinder $\gpq 011$, i.e.\ for a sphere with one ingoing and one outgoing boundary circle,
is invertible. (In the conformal field theory literature, this condition is known
as non-degeneracy of the two-point function of bulk fields,
see e.g. \cite[Thm.\,4.26]{fjfrs2} for a succinct statement.)

Let us rewrite these conditions more compactly. We introduce the constant symmetric 
monoidal functor
that assigns the ground field to each object and the identity morphism to each morphism, 
  $$
  \begin{array}{rcl}
  \One :\quad  \SurfC &\!\!\longrightarrow\!\!& \Vect \Nxl1
  \surf &\!\!\longmapsto\!\!& \ko \Nxl1
  \surf\,{\stackrel\varphi\to}\surf' &\!\!\longmapsto\!\!& \id_\ko
  \eear
  $$
We can then give  

\begin{defi}\label{def:Corr}
Let \D\ be a modular finite ribbon category and $F\iN \D$ an object.
A \emph{consistent system of bulk field correlators} for monodromy data based on \D\ 
and with bulk object $F$ is a monoidal natural transformation
  \be
  \Corr:\quad \One \Rightarrow \BlF
  \labl{eq:Corr}
for which the linear map $(\id_F \oti d_F) \cir (\Corr(\gpq 011) \oti \id_F) \iN \EndD(F)$
is invertible.
\end{defi}

Here and below we identify the linear map $\Corr(\surf)$ with its value at 
$1 \iN \ko \eq \One(\surf)$. 

    \medskip

As in the construction of the pinned bock functor, for examining such natural 
transformations it is advantageous to also work with marked surfaces. Thus we introduce 
the constant functor $\mOne\colon \mSurfC\To\Vect$ that again maps every object to 
\ko\ and every morphism to $\id_\ko$, and consider monoidal natural transformations
  \be
  \mCorr:\quad \mOne \Rightarrow \mBlF .
  \labl{eq:mCorr}
Once we are given such a monoidal natural transformation $\mCorr$, we can obtain a 
corresponding consistent system $\Corr$ of correlators by invoking the defining 
universal property of the Kan extension \erf{eq:theKan}.
Indeed, the two constant functors $\One$ and $\mOne$ are related by a right Kan extension
  $$
  \begin{tikzcd}
  \mSurfC \ar{rr}{\mOne}[name=endNatTrafb,below]{} \ar{d}[swap]{\UC} &~& \Vect
  \\
  \SurfC\, \ar[yshift=-3pt,dashed]{urr}[swap,xshift=-2pt] {\One}
  \ar[Rightarrow,xshift=-5pt,to path=-- (endNatTrafb) \tikztonodes]{}[description]{\idsm}
  \end{tikzcd}
  $$
with trivial natural transformation. When composed with a natural transformation 
$\mCorr\colon \mOne \,{\Rightarrow} 
      $\linebreak[0]$
\mBlF$, this gives the diagram
$$
  \begin{tikzcd}
  \mSurfC \ar{rr}{\mBlF}[name=endNatTrafb,below]{} \ar{d}[swap]{\UC} &~& \Vect
  \\
  \SurfC\, 
  \ar[in=270, out=353]{urr}{\One}
  \ar[Rightarrow,xshift=5pt,to path=-- (endNatTrafb) \tikztonodes]{}[description]{\mCorr}
  \end{tikzcd}
  $$
By the universal property of $\BlF$ as a right Kan extension there then exists,
uniquely up to unique natural isomorphism, a natural transformation 
$\Corr\colon \One \,{\Rightarrow}\, \BlF$ such that $\mCorr$ is given by the composition
\be
  \begin{tikzcd}
  \mSurfC \ar{rr}{\mBlF}[name=endNatTrafb,below]{} \ar{d}[swap]{\UC} &~& \Vect
  \\[8pt]
  \SurfC \ar[ yshift=-3pt]{urr}[name=endCorr,below]{}{}
  \ar[in=270, out=353]{urr}[name=beginCorr]{} {}
  \ar[Rightarrow,xshift=-5pt,to path=-- (endNatTrafb) \tikztonodes]{}[description]{\psi}
  \ar[Rightarrow,to path=(beginCorr) -- (endCorr)] {}
  & ~ & \begin{picture}(0,0) \put(-6,19) {$\scriptstyle \One$}
  \put(-41,19){$\scriptstyle \Corr$}
  \put(-55,34){$\scriptstyle \BlF$} \end{picture}
  \end{tikzcd}
  \labl{eq:Kan4Corr}
with $\psi$  
the natural transformation that is part of the Kan extension \erf{eq:theKan}. 
Also, if the natural transformation $\mCorr$ is monoidal, then so is $\Corr$, and if 
$\mCorr(\gpq 011,\mark)$ for any marking $\mark$ is invertible, then so is $\Corr(\gpq 011)$.

We will refer, analogously as in Definition \ref{def:Corr}, to a natural transformation 
$\mCorr\colon \mOne \,{\Rightarrow}\, \mBlF$ as a \emph{consistent system of correlators on
marked surfaces} or, slightly abusing terminology, just as a consistent system of correlators.


\section{Consistent systems of correlators}\label{sec4}

We are now in a position to address our primary goal: to formulate necessary and sufficient 
conditions for the existence of a consistent system of bulk field correlators with
bulk object $F$, expressed as a monoidal natural transformation \erf{eq:Corr}.
To achieve this, taking advantage of the results of Section \ref{sec3}, we work with
the category \mSurfC\ of marked surfaces instead of \SurfC, and thus consider instead
of \erf{eq:Corr} a monoidal natural transformation $\mCorr$ \erf{eq:mCorr}.
That $\mCorr$ is a natural transformation from $\mOne$ to $\mBlF$ means that the square
  \be
  \begin{tikzcd}
  \ko \ar{d}[left]{\mCorr\surfm} \ar{rr}{\idsm_\ko} &~&
  \ko \ar{d}{\mCorr(\surf',\mark')}
  \\
  \mBlF\surfm \ar{rr}{\mBlF(f)} &~& \BlF(\surf',\mark')
  \end{tikzcd}
  \labl{eq:f-square}
commutes for any morphism $f$ of \mSurfC. Moreover, by the construction of the 
functor $\mBlF$, such a natural transformation is already completely characterized by 
commutativity of this square for all generating morphisms $f$ (see Section \ref{sec:fine}
and formulas \erf{elemsurf1} and \erf{elemsurf2}) of \mSurfC.


\subsection{Elementary correlators}

For analyzing the condition \eqref{eq:f-square} it is worth recalling the basic idea 
of the construction of conformal blocks in Section \ref{sec:mBl}:
First one defines conformal block functors for spheres with at most three holes, and then
these are used as building blocks to obtain the functors for arbitrary surfaces with the 
help of coend constructions. Inspired by this idea, here
we start by selecting vectors in the morphism spaces $\mBlF\surfm$
for $\surf$ a sphere with at most three holes and then obtain vectors in $\mBlF\surfm$
for any marked surface $\surfm$ which are uniquely determined by the requirement to
furnish a monoidal natural transformation $\mCorr\colon \mOne \,{\Rightarrow}\, \mBlF$.

In fact, denoting by $\gpq 0pq$ a sphere with $p$ incoming and $q$ outgoing boundary 
circles, we have 

\begin{prop}\label{prop:003-010-020}
A monoidal natural transformation $\,\mCorr\colon \mOne \,{\Rightarrow}\, \mBlF$
is completely determined by its values on the spheres $\gpq 003$, $\gpq 010$ and $\gpq 020$
with any choice of marking without cuts.
\end{prop}

\begin{proof}
Commutativity of \erf{eq:f-square} for $\surf \eq \surf'$ on these three surfaces and 
for $f$ the Z-isomorphism \erf{def-Z-iso} amounts to the statement that the values of 
$\mCorr$ on these surfaces indeed do not depend on a choice of marking without cuts on them.
Further, that the natural transformation $\mCorr$ is (strictly) monoidal simply means that
  $$
  \mCorr\big( (\surf_1,\mark_{\!1}) \,{\sqcup}\, (\surf_2,\mark_{\!2}) \big)
  = \mCorr(\surf_1,\mark_{\!1}) \otimes_\ko \mCorr(\surf_2,\mark_{\!2}) \,.
  $$
As a consequence we can restrict our attention to connected surfaces.
\\[2pt]
Thus let $(\surf,C,\mark)$ be any connected marked surface with non-empty fine cut system 
$C$ and $c \iN C$ any of the cuts. Then by compatibility with sewing, the correlator for
$(\surf,C,\mark)$ must coincide with the vector obtained from the correlator for
the cut surface $\oline\surf{\{c\}}$ by applying the linear map that by Definition 
\ref{def:mBlF3} is associated to the sewing $\mathsf s\colon \oline\surf{\{c\}} \To \surf$:
  \be
  \mCorr(\surf,C,\mark) = \mBlF(\mathsf s) \circ \mCorr(\oline\surf{\{c\}})
  \labl{eq:mCorr-cut}
(in the notation we suppress the marking of $\oline\surf{\{c\}}$, which is completely 
determined by the one of $\surf$).
Repeating this prescription for every cut in $C$, $\mCorr(\surf,C,\mark)$ gets expressed
through the correlators for spheres that have at most three holes and markings without
cuts. To complete the proof we need to show that all the latter correlators can, in turn,
be expressed through the three correlators $\mCorr(\gpq 003)$, $\mCorr(\gpq 010)$ and 
$\mCorr(\gpq 020)$. This will be done separately in Lemma \ref{lem:anysphere-from-3} below.
\end{proof}

Note that a priori the so defined vectors could depend on the order in which the cuts 
are removed. Recall, however, that according to Definition \ref{def:mBlF3}
the linear maps $\mBlF(\mathsf s)$ are given by the dinatural structure
morphisms of coends (see formulas \erf{eq:sewmap1} and \erf{eq:sewmap2}).
Independence of the ordering is thus in fact guaranteed by the Fubini 
theorems \eqref{eq:Fubini} and \eqref{eq:Fubini2}.

\medskip

To formulate the postponed part of the proof it will be convenient to 
work with the three morphisms
  \be
  \bearl
  \omF := \mCorr(\gpq 003) \,\in \HomD(\one,F\oti F\oti F) \,,
  \Nxl3
  \epF := d_F \cir (\mCorr(\gpq 010) \oti \id_F) \,\in \HomD(F,\one)
  \qquad\quad{\rm and}
  \Nxl3
  \PhF := (\id_{F^\vee_{}}^{} \oti d_F)\cir (\mCorr(\gpq 020) \oti \id_F)
  \,\in \HomD(F,F^\vee_{}) \,.
  \eear
  \labl{om,eps,Phi}


\begin{lem}\label{lem:anysphere-from-3}
Let $\mCorr$ be a natural transformation $\mCorr\colon \mOne \,{\Rightarrow}\, \mBlF$.
For $\surf$ any sphere with at most three holes and $\mark$ any fine marking without cuts
on $\surf$, $\mCorr(\surf,\mark)$ can be expressed through the three morphisms 
{\rm \erf{om,eps,Phi}}. Specifically,
\\[3pt]
{\rm (i)}\, 
To $(\gpq 002,\mark)$, $\mCorr$ assigns the vector
  \be
  (\epF \oti \id_F \oti \id_F) \circ \omF
  = (\id_F \oti \epF \oti \id_F) \circ \omF
  = (\id_F \oti \id_F \oti \epF) \circ \omF 
  \labl{eq:eii-om}
in $\HomD(\one,F\oti F)$.
\\[3pt]
{\rm (ii)}\, $\mCorr(\gpq 001,\mark) \eq (\epF \oti \epF \oti \id_F) \cir \omF$, and similarly 
as in {\rm \erf{eq:eii-om}} one gets the same vector in $\HomD(\one,F)$ if the occurrence 
of $\id_F$ in this expression is interchanged with any occurrence of $\epF$.
\\[3pt]
{\rm (iii)}\, $\mCorr(\gpq 000,\emptyset) \eq (\epF \oti \epF \oti \epF) \cir \omF$.
\\[3pt]
{\rm (iv)}\, $\mCorr(\gpq 011,\mark)$ is given by either $(\id_F \oti \epF \oti \PhF) \cir \omF$ 
or $(\PhF \oti \epF \oti \id_F) \cir \omF$, depending on whether the boundary circle on which 
the distinguished first edge of $\mark$ ends is outgoing or incoming. 
\\[3pt]
{\rm (v)}\, $\mCorr(\gpq 012,\mark)$ is given by either $(\PhF \oti \id_F \oti \id_F) \cir \omF$ 
or $(\id_F \oti \PhF\oti \id_F) \cir \omF$ or $(\id_F \oti \id_F 
     $\linebreak[0]$
\oti \PhF) \cir \omF$,
depending on whether the edge that ends on the incoming boundary circle 
is the first, second or last of the three edges of $\mark$.
\\[3pt]
{\rm (vi)}\, $\mCorr(\gpq 021,\mark)$ is given by either $(\id_F \oti \PhF \oti \PhF) \cir \omF$ 
or $(\PhF \oti \id_F \oti \PhF) \cir \omF$ or $(\PhF\oti \PhF 
     $\linebreak[0]$
\oti \id_F) \cir \omF$,
depending on whether the edge that ends on the outgoing boundary circle 
is the first, second or last of the three edges of $\mark$.
\\[3pt]
{\rm (vii)}\, $\mCorr(\gpq 030,\emptyset) \eq (\PhF \oti \PhF \oti \PhF) \cir \omF$.
\end{lem}


\begin{proof}
(i)\, Both the equalities \erf{eq:eii-om} and the assertion that $\mCorr(\gpq 002,\mark)$
is given, for any $\mark$, by this vector follow by combining invariance under 
the Z-isomorphism and compatibility with sewing morphisms
$\mathsf s\colon \gpq 003 \,{\sqcup}\, \gpq 010 \,{\to}\, \gpq 002$.
To see this, one notes that according to Definition \ref{def:mBlF3},
$\mBlF(\mathsf s)$ acts on $\mCorr(\gpq 003) \oti \mCorr(\gpq 010)$ as 
post-composition by $\id_F \oti [ d_{F^\vee_{}}^{} \cir (\pi_F \oti \id_{F_{}}^{}) ]$,
and by the defining properties of the pivotal structure we have 
$d_{F^\vee_{}}^{} \cir (\pi_F \oti \mCorr(\gpq 010)) \eq \epF$.
\\
Note that in the first place sewing $\gpq 003$ to $\gpq 010$ yields a marking on 
$\gpq 002$ that has a cut. But by compatibility with the F-move (M3) this cut can
be omitted without changing the value of $\mCorr(\gpq 002)$.
\\[2pt]
(ii)\, is shown in the same way as (i), considering instead sewings
$\gpq 002 \,{\sqcup}\, \gpq 010 \,{\to}\, \gpq 001$.
\\[2pt]
(iii)\, follows from (ii) by compatibility with the sewing
$\gpq 001 \,{\sqcup}\, \gpq 010 \,{\to}\, \gpq 000$. 
\\[2pt]
(iv)\,--\,(vii) follow in the same way as (i) when considering sewings
$\gpq 002 \,{\sqcup}\, \gpq 020 \,{\to}\, \gpq 011$, respectively
$\gpq 003 \,{\sqcup}\, \gpq 020 \,{\to}\, \gpq 012$, respectively
$\gpq 012 \,{\sqcup}\, \gpq 020 \,{\to}\, \gpq 021$, respectively
$\gpq 021 \,{\sqcup}\, \gpq 020 \,{\to}\, \gpq 030$.
\end{proof}


\subsection{Correlators on surfaces of genus zero}

In view of Proposition \ref{prop:003-010-020} it is not hard to examine the implications
of the naturality requirement \erf{eq:f-square}. In particular, commutativity of 
\erf{eq:f-square} for $f \eq \mathsf s$ a sewing amounts to well-definedness of the
prescription given in the proof of Proposition \ref{prop:003-010-020} which,
as already noted, is ensured by the Fubini theorem \eqref{eq:Fubini}.
We are thus left with the case that $f$ is an elementary move of one of the
groupoids $\CWm(\surf)$ -- or rather, to be precise, either an admissible elementary
move or a simple combination of non-admissible elementary moves that is admissible. In
addition we must take care of the non-degeneracy requirement in Definition \ref{def:Corr}. 

In the present subsection we restrict our attention to moves involving genus zero surfaces 
only. Accordingly (compare Remark \ref{rem:g=0}) in this subsection we only assume that 
\D\ is a finite ribbon category.
Besides being a natural preparation for the general case, this restricted situation is
also of interest in applications of conformal field theories to critical phenomena in
statistical mechanics, as well as when discussing operator product expansions
\cite[Sect.\,4]{bepz} and correspondingly in the study of vertex operator algebras (see
e.g.\ \cite{huan7}).
Recall that the functor $\mBlF$ maps the elementary moves to the isomorphisms
\erf{def-Z-iso} (Z-isomorphism), \erf{def-B-iso} (B-iso\-morphism),
\erf{def-F-isoUV} (F-isomorphism) and \erf{def-A-iso} (A-isomorphism). We first show

\begin{lem}\label{lem:selfdual-etc}
Let $\mCorr$ be a consistent system of correlators on marked surfaces. Then
\\[3pt]
{\rm (i)}\,
The morphisms $\omF$ and $\epF$ defined in {\rm \erf{eq:eii-om}} are non-zero. 
\\[3pt]
{\rm (ii)}\,
The morphism $\PhF$ defined in {\rm \erf{eq:eii-om}} is an isomorphism.
\\[3pt]
{\rm (iii)}\,
$\omF \iN \HomD(\one,F^{\otimes 3})$ is invariant under any braiding of the three
$F$-strands.
\\[3pt]
{\rm (iv)}\,
The object $F \iN \D$ is self-dual and has trivial twist.
\end{lem}

\begin{proof}
On the cylinder $\gpq 011$ select a marking $\mark$ without cuts 
that has trivial winding with respect to the non-contractible cycle of the cylinder 
and for which the distinguished edge ends on the outgoing boundary circle. Abbreviate
$j_F \,{:=}\, (\id_F \oti d_F) \cir [ \mCorr(\gpq 011,\mark) \oti \id_F ] \iN \EndD(F)$.
The non-degeneracy condition demands that $j_F$ is invertible. On the other hand,
gluing two such cylinders and invoking invariance under the F-isomorphism shows that
$j_F$ is an idempotent; thus $j_F \eq \id_F$. This is equivalent to
$ \mCorr(\gpq 011,\mark) \eq b_F $, from which it follows in particular that $\omF$ and 
$\epF$ are non-zero, thus proving (i), and that the morphism
$\PhF^- \,{:=}\, (d_F \oti \epF \oti \id_F) \cir (\idFv \oti \omF) 
   $\linebreak[0]$
   {\in}\, \HomD(F^\vee,F)$
is a right-inverse of $\PhF$.
\\ 
Next we note, using naturality of the braiding and the relation between the pivotal
structure $\pi$ and the twist, that for any marking $\mark$ on $\gpq 002$ we have
  $$
  \mZ_{F,F}\big(\mCorr(\gpq 002,\mark)\big) 
  = (\id_F \oti \theta_F) \circ \big( \mB_{F,F,\one}\big(\mCorr(\gpq 002,\mark)\big) \big) \,.
  $$
Invariance of $\mCorr(\gpq 002,\mark)$ under both the Z- and the B-isomorphism thus implies
that $F$ has trivial twist, $\theta_F \eq \id_F$.
Using this result, it follows further that $\mZ_{F,F\otimes F}$ and $\mB_{F,F,F}$
generate an action of the braid group $B_3$ on the three tensor factors $F$ in the codomain
of $\omega$. Invariance of $\mCorr(\gpq 003,\mark)$ under the Z- and the B-isomorphisms thus 
proves the claim (iii).
Combining the results obtained so far one easily shows that the morphism $\PhF^-$ is 
also a left-inverse of $\PhF$, thus completing the proof of (ii).
Finally, (ii) implies that $F$ is self-dual.
\end{proof}

One can further check that
  $$
  (d_F \oti \epF \oti \id_F) \cir (\idFv \oti \omF) \equiv \PhF^{-1}
  = \big(\id_F \oti \epF \oti [ d_{F^\vee_{}}^{} \cir (\pi_F \oti \idFv) ] \,\big)
  \circ (\omF  \oti \idFv) \,,
  $$
which in turn implies
  \be
  \PhF^\vee \circ \pi_F^{} = \Phi_F \,.
  \labl{eq_Phi.pi=Phi}
This identity (which was in fact already used implicitly in the statement of
Lemma \ref{lem:anysphere-from-3}\,(v) and (vi) above) means that the object $F$ 
is not only self-dual, but also 
that its Frobenius-Schur indicator (see e.g.\ \cite[Sect.\,2.3]{fffs3}) is equal to 1.

To proceed we introduce the specific expressions
  \be
  \bearl
  \Delta_F := \big( \id_{F^{\otimes 2}_{}}^{} \oti [ d_F \cir (\PhF \oti \id_F) ]\,\big)
  \circ (\omF \oti \id_F) ~\in \HomD(F,F\oti F) \,,
  \Nxl3
  \eta_F := (\epF \oti \epF \oti \id_F) \circ \omF ~\in \HomD(\one,F)  \qquad{\rm and}
  \Nxl3
  m_F := \big( \id_F \oti [ d_{F^{\otimes 2}_{}}^{} \cir (\PhF \oti \PhF \oti
  \id_{F^{\otimes 2}_{}}^{} )]\,\big) \circ (\omF \oti \id_{F^{\otimes 2}_{}}^{})
  ~\in \HomD(F\oti F,F) 
  \eear
  \labl{eq:def-Delta-m-eta}
in the morphisms used so far. We have

\begin{prop}\label{prop:Frob}
Let $\mCorr$ be a consistent system of correlators with bulk state space $F$. Then 
the morphisms $(m_F,\eta_F,\Delta_F,\epF)$ endow the object $F$ with the structure of
a Frobenius algebra in \D.
\end{prop}

\begin{proof}
Invariance of $\mCorr(\gpq 004,\mark)$, for a suitable choice of marking $\mark$, under 
the A-isomorphism amounts to the equality
  \be
  \bearl
  \big( \id_{F^{\otimes 2}_{}}^{} \oti [ d_F \cir (\PhF \oti \id_F) ]
  \oti \id_{F^{\otimes 2}_{}}^{} \big) \circ (\omF \oti \omF)
  \Nxl3 \hsp5
  = \big( \id_{F^{\otimes 3}_{}}^{} \oti [ d_F \cir (\PhF \oti \id_F) ] \oti \id_F \big)
  \circ (\id_F \oti \omF \oti \id_{F^{\otimes 2}_{}}^{}) \circ \omF \,.
  \eear
  \labl{eq:A-inv}
When expressed in terms of $\Delta_F$, this is nothing but coassociativity. 
The counit properties of $\epF$ are equivalent to the result obtained in the proof of
Lemma \ref{lem:selfdual-etc} that the morphism $j_F$ defined there equals $\id_F$.
This shows that $(F,\Delta_F,\epF)$ is a coalgebra.
That $(F,m_F,\eta_F)$ is an algebra follows by dualizing these considerations after
noticing that
  \be
  m^\vee_F = (\PhF^{} \oti \PhF^{}) \circ \Delta_F^{} \circ \PhF^{-1} \qquand
  \eta_F^\vee = \epF^{} \circ \PhF^{-1} .
  \labl{eq:mvee}
Finally, the Frobenius property is seen as follows. From the definition
\eqref{eq:def-Delta-m-eta} of the product and coproduct it follows in particular that
we also have
  $$
  m_F = (\id_F \oti d_F) \circ (\id_F \oti \PhF \oti \id_F) \circ (\Delta_F \oti \id_F) \,,
  $$
which together with the counit property of $\epF$ implies that
$\PhF \eq ( (\epF\cir m_F) \oti \idFv) \cir (\id_F \oti b_F)$. That $\PhF$ is invertible
is thus equivalent to the statement that the bilinear form
$\kappa_F \,{:=}\, \epF \cir m_F \,{=}  
           $\linebreak[0]$
d_F \,{\circ}\, (\PhF \oti \id_F)$ in $\HomD(F\oti F,\one)$ is non-degenerate. Also,
owing to associativity of $m_F$, the form $\kappa_F$ is invariant, and thus it constitutes a 
Frobenius form for $F$.
\end{proof}

\begin{rem}\label{rem:graph}
The manipulations in the proofs of Lemma \ref{lem:selfdual-etc} and Proposition \ref{prop:Frob} 
can conveniently be performed with the help of a graphical calculus for morphisms in monoidal
categories. We content ourselves to illustrate this with a few representative calculations.
First, depicting the basic morphisms \eqref{om,eps,Phi} as
  \eqpic{11} {320} {36} {
  \put(1,19)    {$ \omF ~=: $}
  \put(46,0)    {\includepic{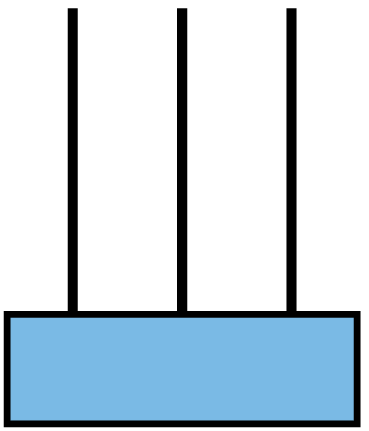} }
  \put(140,19)  {$ \epF^{} ~=: $}
  \put(182,9)   {\includepic{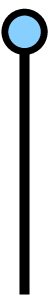} }
  \put(250,19)  {$ \PhF ~=: $}
  \put(296,0)   {\includepic{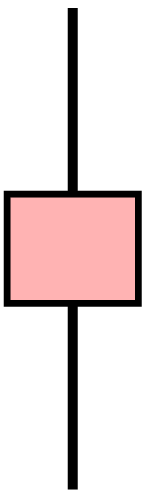} 
  } }
the equality \eqref{eq:A-inv} that implements the compatibility of the correlators 
$\mCorr(\gpq 004,\mark)$ with the A-iso\-mor\-phism reads
  \eqpic{38} {190} {79} {
  \put(0,0)     {\includepic{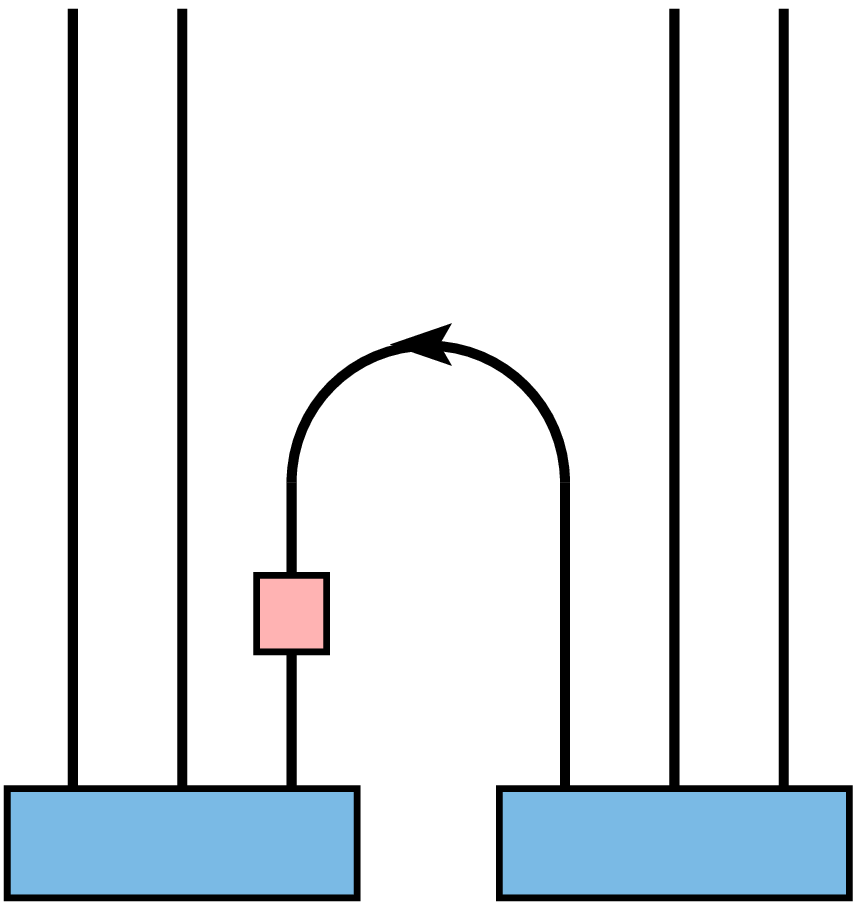} }
  \put(99,38)   {$ = $}
  \put(134,0)   {\includepic{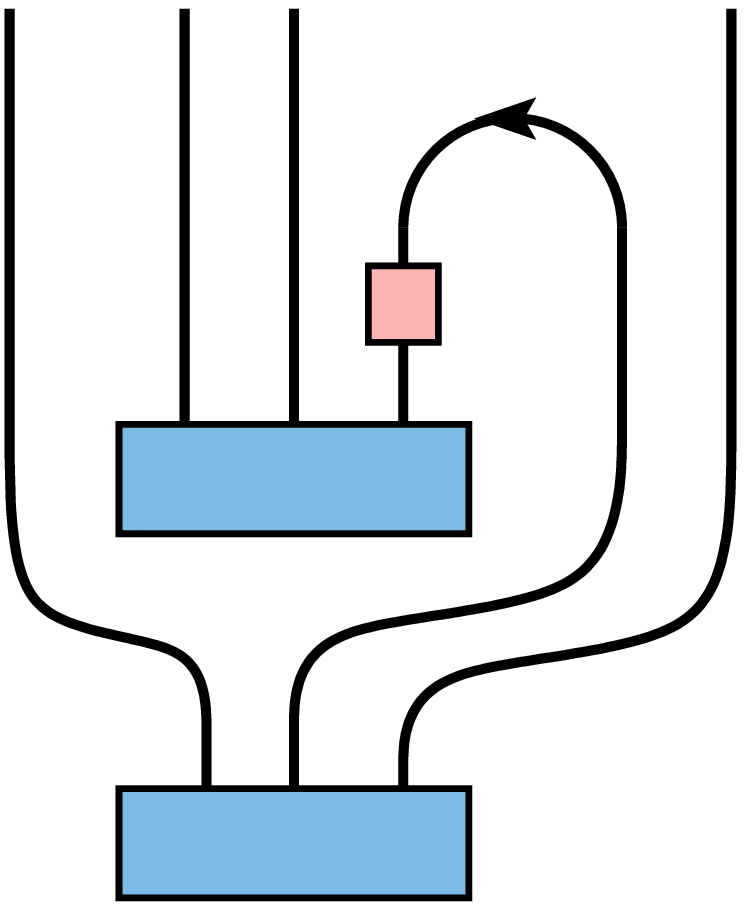}
  } }
Noticing that
  \eqpic{39} {110} {65} {
  \put(0,29)    {$ \Delta_F ~= $}
  \put(50,0)    {\includepic{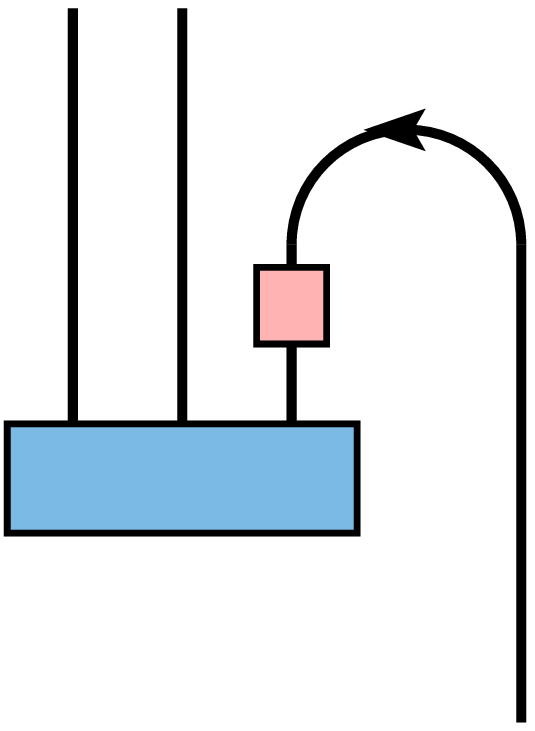}
  } }
this equality expresses the coassociativity of $\Delta_F$. Second, the relation 
between product and coproduct given in \eqref{eq:mvee} is obtained by the sequence
  \eqpic{41} {268} {104} {
  \put(-96,47) {$ m_F ~:= $}
  \put(-46,13)  {\includepic{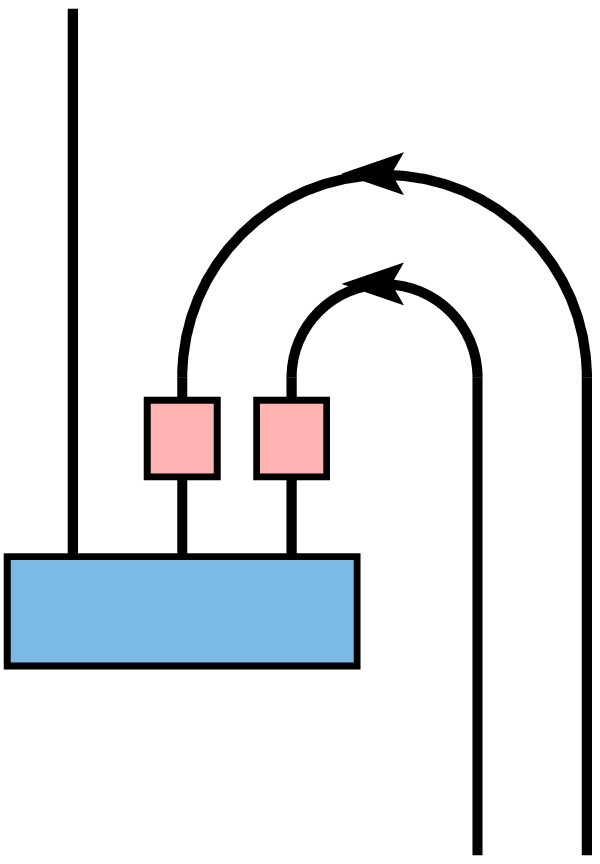} }
  \put(24,47)   {$ = $}
  \put(50,0)    {\includepic{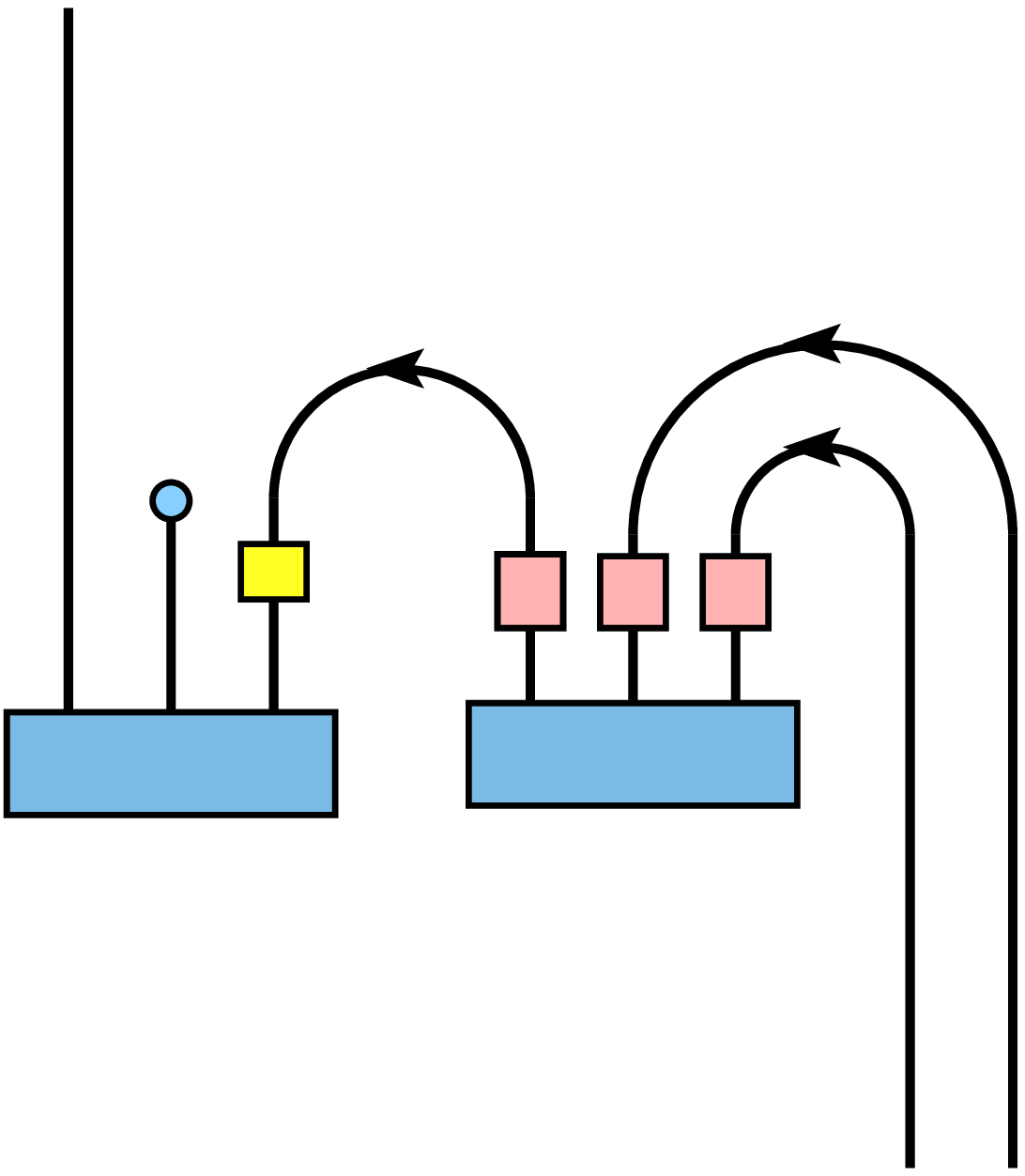} }
  \put(162,47)  {$ = $}
  \put(188,0)   {\includepic{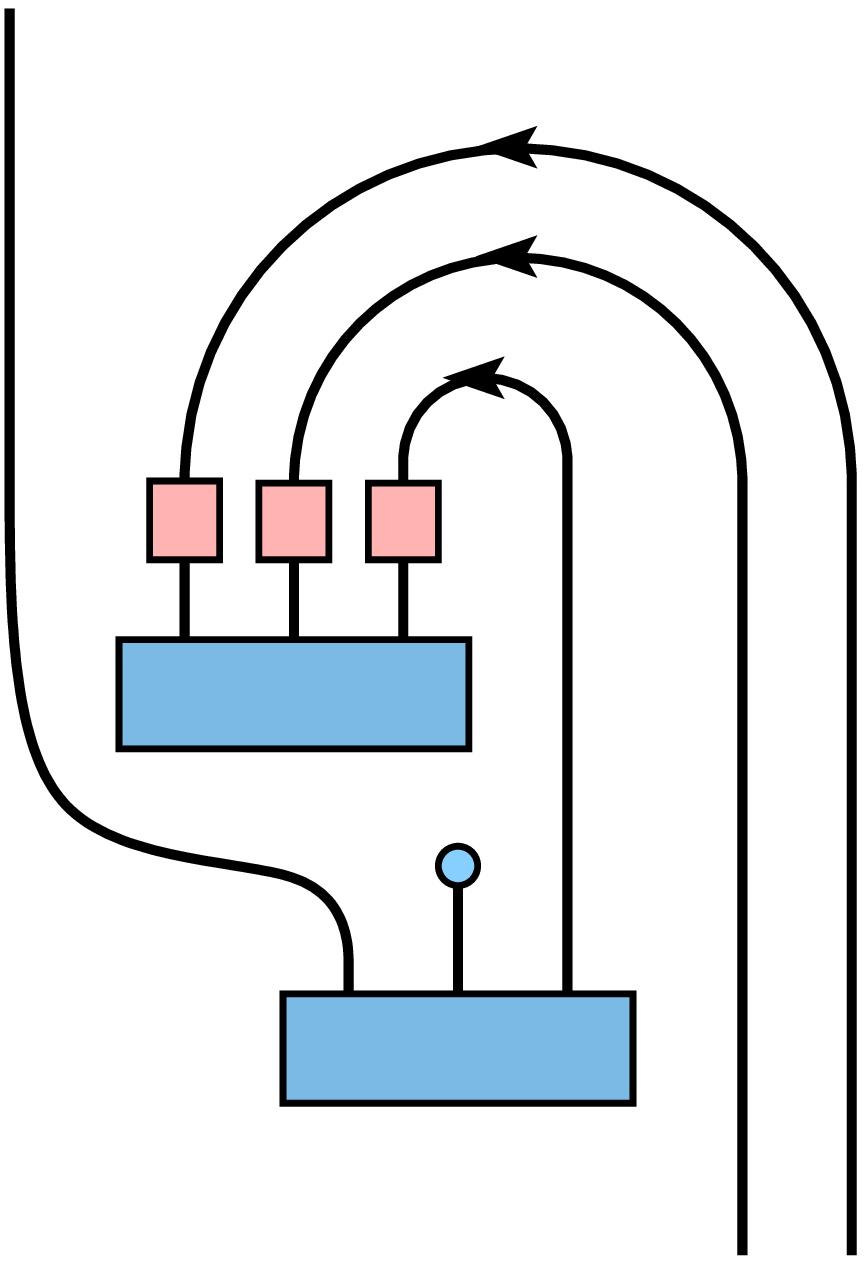} }
  \put(279,47)  {$ = $}
  \put(306,0)   {\includepic{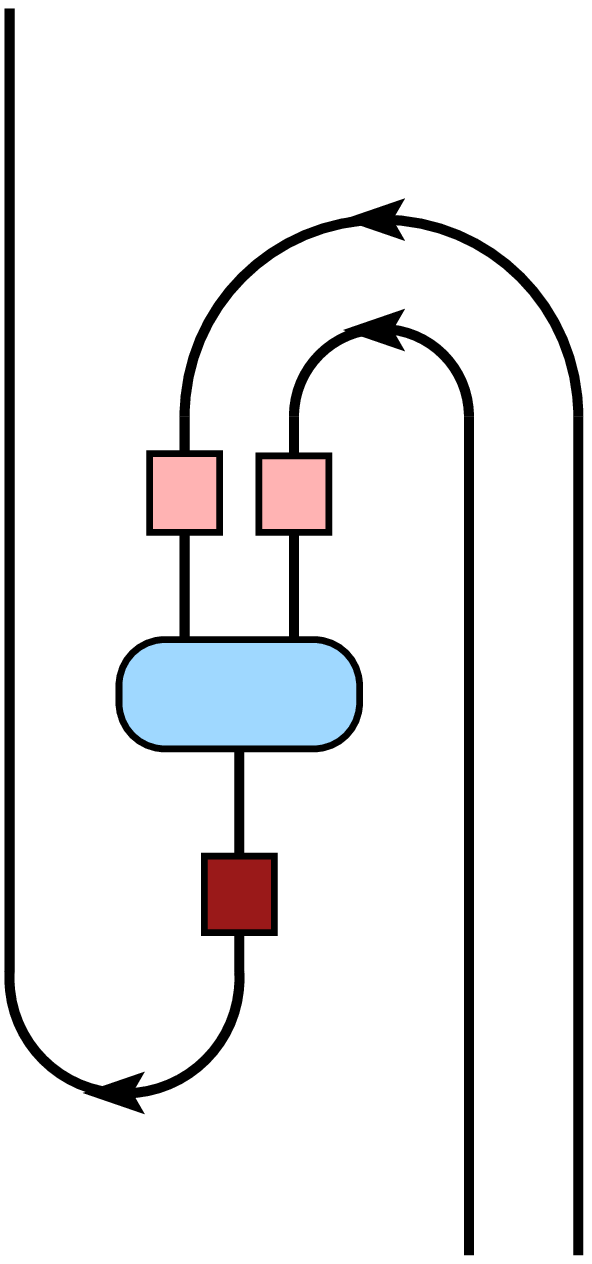}
  \put(5,33.3)  {\sse$ \PhF^{-\!1} $}
  \put(15,47.7) {\sse$ \Delta_F $}
  } }
of equalities. And third, the inverse of $\PhF$ as introduced in the proof of
Lemma \ref{lem:selfdual-etc} satisfies 
  \eqpic{16} {210} {63} {
  \put(0,29)    {$ \PhF^{-1} ~:= $}
  \put(54,0)    {\includepic{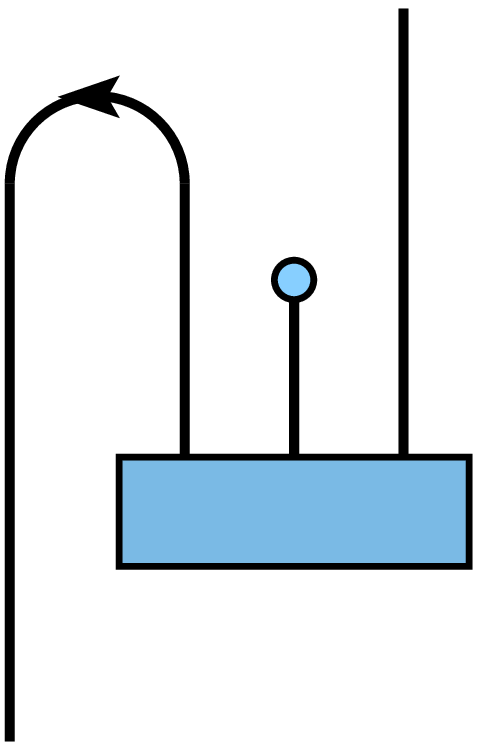}   }
  \put(120,29)  {$ = $}
  \put(151,0)   {\includepic{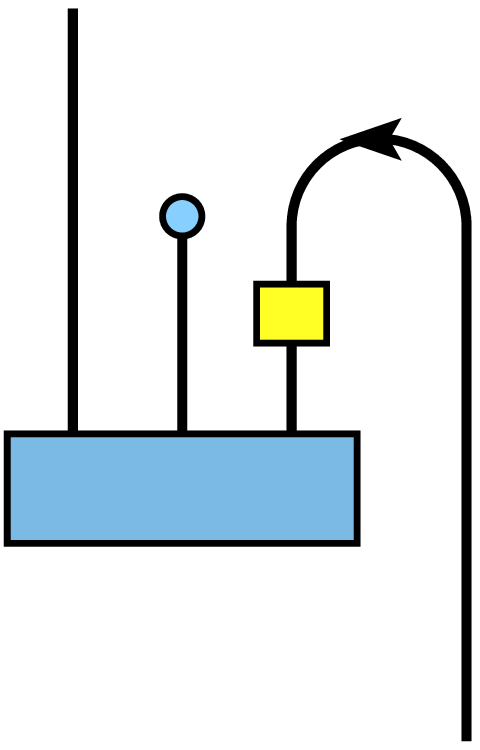} 
  \put(30.8,35.7) {\sse$ \pi_{\!F} $}
  } }
which is e.g.\ used in deriving the relation \eqref{eq_Phi.pi=Phi}
between $\PhF^{}$ and $\PhF^\vee$.
\end{rem}    

Further properties of $F$ now follow easily:

\begin{prop}
Let $\mCorr$ be a consistent system of correlators with bulk state space $F$. Then the
Frobenius algebra $(F,m_F,\eta_F,\Delta_F,\epF)$ is {\rm(}co{\rm)}commutative and symmetric.
\end{prop}

\begin{proof}
Cocommutativity of $\Delta_F$ is an immediate consequence of the invariance of $\omF$
under the B-isomorphism. For a Frobenius algebra, commutativity is equivalent to
cocommutativity. (In the case at hand, commutativity of $m_F$ also follows directly by 
dualizing after invoking \erf{eq:mvee}.) Commutativity together with triviality of the twist
imply that the Frobenius form $\epF \cir m_F$ is symmetric in the sense of ribbon categories. 
(Note that the property of being (ribbon) symmetric does not include commutativity; 
conversely, $\theta_F \eq \id_F$ follows by combining commutativity and symmetry.)
\end{proof}

It is not hard to check that demanding that $\mCorr$ is a consistent system of correlators
on marked surfaces of \emph{genus zero} does not lead to any further constraints
than those which we have treated already. Moreover, by reading some of 
the arguments backwards it is evident that the restrictions on the object $F$ obtained 
above are indeed also sufficient. Recalling in addition that $\mCorr$ gives us a monoidal 
natural transformation $\Corr\colon \One \,{\Rightarrow}\, \BlF$ via \erf{eq:Kan4Corr}, 
we can state the first part of our main result, which describes the bulk fields of 
a conformal field theory that is defined on surfaces of genus zero:

\begin{prop}\label{prop-0}
For \D\ a finite ribbon category, the consistent systems of bulk field correlators on
surfaces of genus zero with {\rm(}genus-zero{\rm)} monodromy data based on \D\ and with
bulk object $F \iN \D$ 
are in bijection with structures of a commutative symmetric Frobenius algebra on $F$.
\end{prop}


\subsection{Higher genus correlators}

To extend our findings to surfaces of any genus we need to analyze invariance under
the S-iso\-morphism. Recall that for $U \iN \D$ the S-isomorphism $\mS_U$ is the linear 
endomorphism of $\HomD(U,K)$ given by post-composition with the isomorphism 
$S^K \iN \EndD(K)$ defined in \erf{SK}.
We are thus dealing with correlators for genus-1 surfaces. Indeed invariance
under the S-iso\-morphism boils down to
invariance of the correlator of a one-holed torus $\gpq 110$
with some choice of marking $\mark$, 
or rather, to fit with the conventions chosen for the S-isomorphism, of the combination
  $$
  \widetilde{\mathrm v}^1_{1|0}
  := (\id_K \oti d_F) \circ (\mCorr(\gpq 110,\mark) \oti \id_F) ~\iN \HomD(F,K) \,.
  $$
In short, a system of correlators that is consistent at genus zero can be consistently
extended to higher genus if and only the equality 
  \be
  S^K \circ \widetilde{\mathrm v}^1_{1|0} = \widetilde{\mathrm v}^1_{1|0}
  \labl{eq:SK.v110=v110}
holds.

Via the prescription \erf{eq:mCorr-cut} the vector $\widetilde{\mathrm v}^1_{1|0}$ is
expressed through the correlator of a three-holed sphere obtained
as a cut surface, $\gpq 021 \eq \oline{\gpq 110}{\{c\}}$.
Further, without loss of generality we may assume that the fine cut system on the torus 
$\gpq 110$ is minimal and thus consists of a single cut, and take the graph $\mark$ on
the torus in such a way that the resulting graph $\mark_{\!c}$ on the cut surface is the
one for which $\mCorr(\gpq 021,\mark_{\!c}) \eq (\id_F \oti \PhF \oti \PhF) \cir \omF$.
Doing so we arrive at
  $$
  \widetilde{\mathrm v}^1_{1|0}
  = (\iK_F \oti d_F) \circ (\id_F \oti \PhF \oti \PhF \oti \id_F) \circ (\omF \oti \id_F)
  $$
with $\iK$ the dinatural transformation of the coend $K$.
Furthermore, upon invoking the universal property of the functor $\BlF$ as a right 
Kan extension as presented in the diagram \eqref{eq:Kan4Corr}, we may identify
$\widetilde{\mathrm v}^1_{1|0}$ with ${\mathrm v}^1_{1|0}$. By comparison with the 
definition of the coproduct $\Delta_F$ in \erf{eq:def-Delta-m-eta} we can then write
  \be  
  {\mathrm v}^1_{1|0} = \iK_F \circ (\id_F \oti \PhF) \circ \Delta_F \,.
  \labl{eq:mCorr110}

\begin{rem}    
In terms of the graphical calculus mentioned in Remark \ref{rem:graph},
the relation \eqref{eq:mCorr110} arises as follows. We have
  \eqpic{42} {360} {99} {
  \put(-2,42)   {$ \mCorr(\gpq 110,\mark) ~= $}
  \put(86,0)    {\includepic{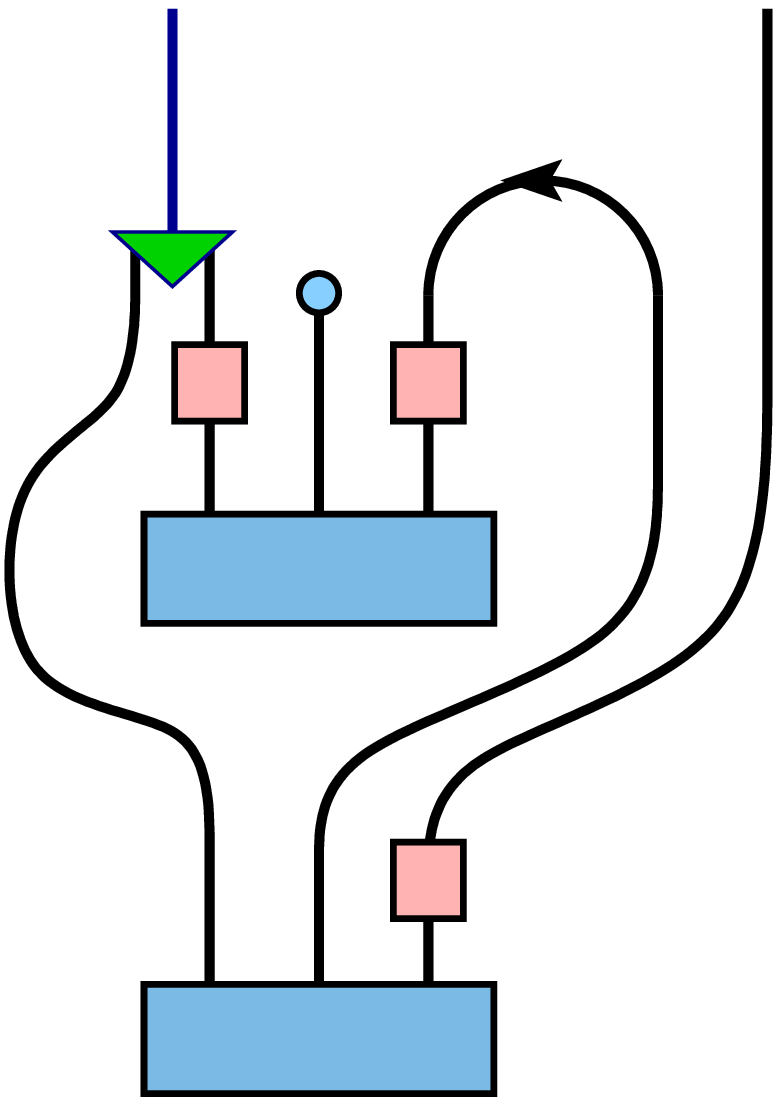}
  \put(.2,72.6) {\sse$ \iK_F $}
  \put(11.7,100){\sse$ K $} }
  \put(172,42)  {$ = $}
  \put(202,0)    {\includepic{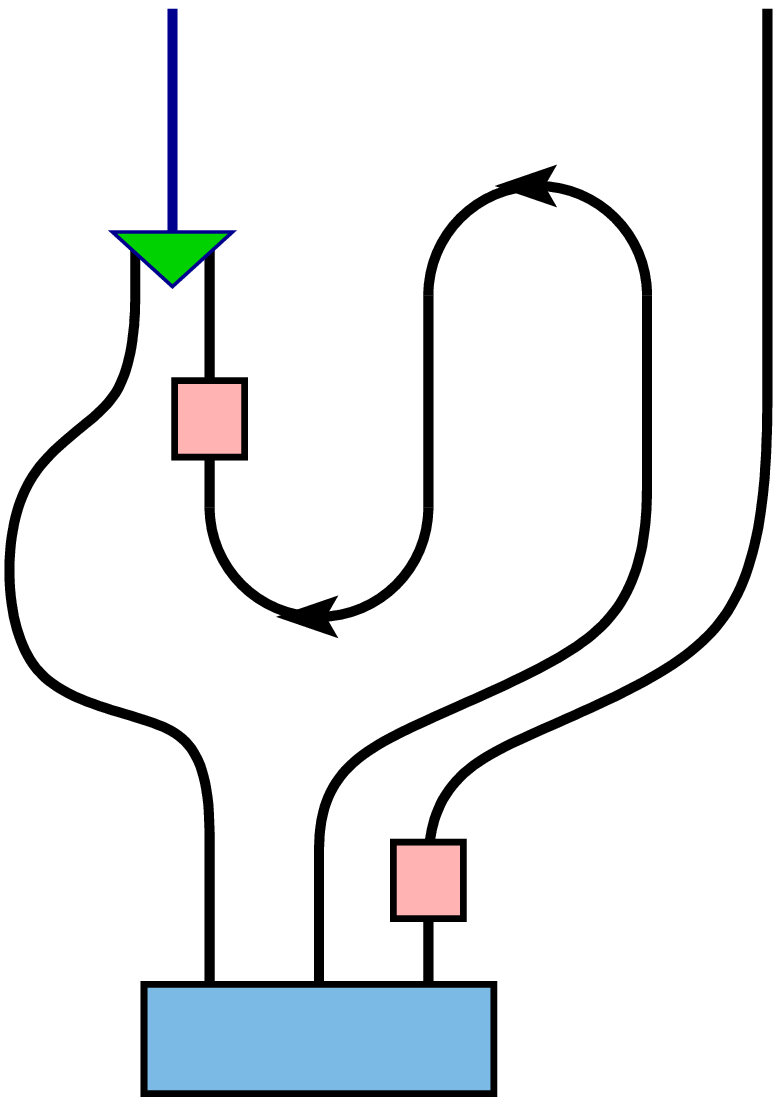} }
  \put(290,42)  {$ = $}
  \put(320,0)   {\includepic{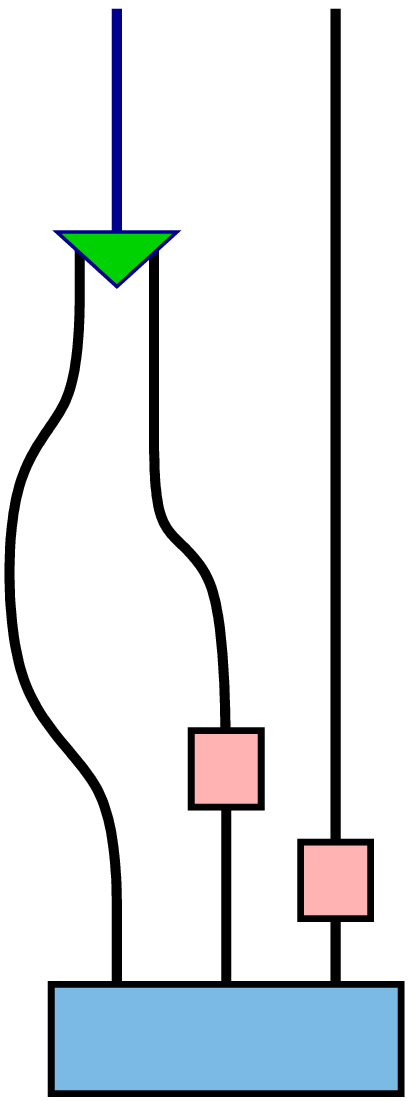} } }
where the first equality holds by construction and the second uses
that $\mCorr(\gpq 011,\mark') \eq b_F$. Thus
  \eqpic{44} {200} {100} {
  \put(0,44)    {$ \widetilde{\mathrm v}^1_{1|0} ~= $}
  \put(56,0)    {\includepic{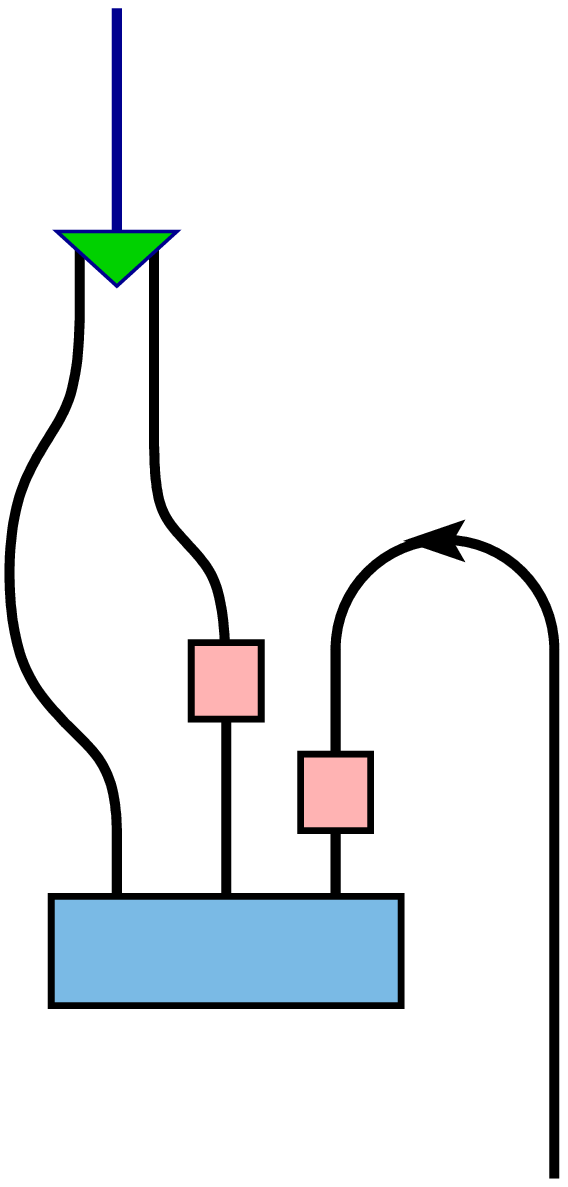} }
  \put(127,44)  {$ = $}
  \put(158,0)   {\includepic{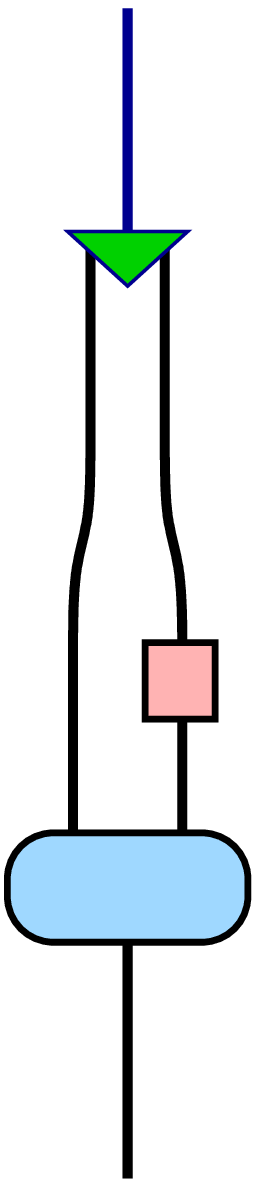}
  \put(5,24.2)  {\sse$ \Delta_F $} } }
\end{rem}    

It follows that invariance under the S-isomorphism can be expressed
as a compatibility property of the coproduct of $F$ with the 
structural morphism $\iK_F$ of the coend $K$. We accommodate this observation by

\begin{defi}\label{def:modFrob}
A commutative symmetric Frobenius algebra $(X,m,\eta,\Delta,\eps)$ in
a modular finite ribbon category \D\ is called \emph{modular} iff it satisfies 
  $$
  S^K \circ [ \iK_X \cir (\id_X \oti \varPhi) \cir \Delta ]
  = \iK_X \cir (\id_X \oti \varPhi) \cir \Delta \,,
  $$
with $\varPhi \eq ((\eps \cir m) \oti \idXv) \cir (\id_X \oti b_X)$.
\end{defi}

We have thus arrived at the second part of our main result:

\begin{thm}\label{thm-g}
For \D\ a modular finite ribbon category, the consistent systems of bulk field correlators 
with monodromy data based on \D\ and with bulk object $F \iN \D$ 
are in bijection with structures of a modular Frobenius algebra on $F$.
\end{thm}

It is worth noting that by combining the compatibility relation \erf{eq:mCorr-cut} 
with the explicit form of $\mBlF(\mathsf s)$ for the relevant sewings $\mathsf s$ 
(see Definition \ref{def:mBlF3}), every morphism
$\mCorr\surfm$ can be expressed in closed form in terms of the three morphisms 
\erf{om,eps,Phi}, the duality and pivotality morphisms for $F$ and the morphism $\iK_F$. 
This expression takes a particularly suggestive form when using the duality to rewrite 
(after again invoking the universal property of the right Kan extension to obtain expressions
involving $\Corr$ rather than $\mCorr$) the correlator for $\gpq gpq$ as a morphism 
$\mathrm v^g_{p|q} \iN \HomD(F^{\otimes p},F^{\otimes q} \oti K^{\otimes g})$ 
and using the following abbreviations: Write
  $$
  \tau_F^{} := (m_F \oti \iK_F) \circ \big(\id_F \oti
  [ (\Phinv \oti \pi_F^{-1}) \cir b_{F^\vee_{}}^{} ] \oti \PhF \big) \circ \Delta_F
  ~\in \Hom(F,F\oti K) \,:
  $$
set $m_F^{(0)} \,{:=}\, \eta_F^{}$, $m_F^{(1)} \,{:=}\, \id_F^{}$, $m_F^{(2)} \,{:=}\, m_F^{}$
as well as $\Delta_F^{(0)} \,{:=}\, \epF^{}$, $\Delta_F^{(1)} \,{:=}\, \id_F^{}$,
$\Delta_F^{(2)} \,{:=}\, \Delta_F^{}$,
and similarly $\tau_F^{(0)} \,{:=}\, \id_F^{}$ and $\tau_F^{(1)} \,{:=}\, \tau_F^{}$;
then define recursively
  $$
  \bearl
  m_F^{(n)} := m_F^{} \circ (m_F^{(n-1)} \oti \id_F) \,, \qquad
  \Delta_F^{(n)} := (\Delta_F^{(n-1)} \oti \id_F) \circ \Delta_F^{} \qquand
  \Nxl3
  \tau_F^{(n-1)} := (\tau_F^{(n-2)} \oti \id_K) \circ \tau_F^{} 
  \eear
  $$
for $n \,{\ge}\, 3$. Then we have

\begin{prop}
Let $\Corr$ be a consistent system of bulk field correlators with bulk object $F$.
Then the correlator for a genus-$g$
surface with $p$ incoming and $q$ outgoing boundary circles is given by
  \be
  \mathrm v^g_{p|q} = (\Delta_F^{(q)} \oti
  \id_{K^{\otimes g}_{\phantom|}}^{}) \circ \tau_F^{(g)} \circ m_F^{(p)} .
  \labl{eq:Delta.tau.m}
\end{prop}

\begin{rem} ~
\\[2pt]
(i)\, For any finite ribbon category the monoidal unit $\one$ carries a trivial structure
of a Frobenius algebra, which is commutative and symmetric. According to Proposition 
\ref{prop-0} it thus provides a consistent system of bulk field correlators at genus zero, 
albeit a rather boring one. In fact, in this case the expression \erf{eq:Delta.tau.m} 
reduces to $\mathrm v^g_{p|q} \eq \tau_\one^{(g)} \eq \eta_K^{\otimes g}$. In particular,
$S^K \cir \mathrm v^1_{1|0} \eq S^K {\circ}\, \eta_K \eq \Lambda_K$, implying that 
the monoidal unit of a modular finite ribbon category \D\
is \emph{not} modular, unless $\D \simeq \Vect$.
\\[2pt]
(ii)\, Existence of any modular Frobenius algebra in a modular finite ribbon category
is far from guaranteed. 
In case \D\ is semisimple, it follows from Corollary 4.1(i) of \cite{dmno} that
a modular Frobenius algebra in \D\ exists iff \D\ is a Drinfeld center.
\\[2pt]
(iii)\, If $\D \eq \C \boti \C^{\rm rev}$ is the enveloping category of a semisimple 
modular tensor category \C, consistent systems of correlators (both 
of bulk fields and of 
boundary fields, and also for surfaces with a network of topological defect lines)
can be constructed with the help of the three-di\-men\-si\-o\-nal 
topological field theory based on \C\ \cite{fuRs4,fjfrs}. The corresponding modular
Frobenius algebras in $\C \boti \C^{\rm rev}$ are related by a center construction
\cite{ffrs,davy20} to the special
symmetric Frobenius algebras in \C\ used in \cite{fuRs4,fjfrs}. Theorem \ref{thm-g} 
reproduces in this case the classification results in \cite{fjfrs2} and \cite{kolR}.
\\[2pt]
(iv)\, For any finite-dimensional factorizable ribbon Hopf algebra $H$, the category 
$H$-mod of fi\-ni\-te-dimensional $H$-modules is a modular finite ribbon category. Further
\cite{fuSs3}, for any ribbon automorphism $\omega$ of $H$ the coend 
$F_\omega \,{:=} \int^{M \in H\text{-mod} \!} \bar\omega(M) \boti M^\vee$, with $\bar\omega$ 
the automorphism of the identity functor induced by $\omega$, carries a
structure of a modular Frobenius algebra in the enveloping category
$H\text{-mod} \boti H\text{-mod}^{\rm rev}$. The formula \erf{eq:Delta.tau.m} for the
associated correlators has in these cases been given in \Cite{Rem.\,3.3}{fuSs5}
(see also \Cite{Eq.\,(3.5)}{fuSs5} for a graphical description). 
It is worth noting that for non-semisimple $H$ the objects 
$F_\omega \iN H\text{-mod} \boti H\text{-mod}^{\rm rev}$ are neither semisimple nor 
projective; we expect that this is a generic feature when \D\ is non-semisimple.
\\[2pt]
(v)\,
For \C\ any finite tensor category, denote by $R$ the right adjoint of the forgetful functor
from the Drinfeld center $\mathcal Z(\C)$ to \C. If \C\ is unimodular, then $R(\one)$ is a 
commutative symmetric Frobenius algebra in $\mathcal Z(\C)$ \Cite{Thm.\,6.1}{shimi7}. 
For $\C \eq H$-mod as in (iv), $R(\one)$ is the image of $F_{\idsm}$ under the equivalence
$\C\boti\C^{\rm rev} \,{\xrightarrow{\,\simeq\,}}\, \mathcal Z(\C)$. It is thus natural to 
conjecture that $R(\one)$ is in fact a modular Frobenius al\-ge\-bra for any modular 
finite ribbon category \C. If \C\ is semisimple, then by Proposition 4.8 of \cite{dmno}
$R(\one)$ is a Lagrangian algebra in the sense of \Cite{Def.\,4.2}{dmno}.
\end{rem}

   \vskip 4em

\noindent
{\sc Acknowledgements:}\\[.3em]
We are grateful to Roald Koudenburg for a correspondence on Kan extensions 
of monoidal functors, and to Kenichi Shimizu for a correspondence on 
Frobenius structures from ambidextruous adjoints.
We also thank the referee for numerous suggestions for improving the exposition.
\\
JF is supported by VR under project no.\ 621-2013-4207.
CS is partially supported by the Collaborative Research Centre 676 ``Particles,
Strings and the Early Universe - the Structure of Matter and Space-Time'', by the RTG 1670
``Mathematics inspired by String theory and Quantum Field Theory'' and by the DFG Priority
Programme 1388 ``Representation Theory''.

\newpage

 \newcommand\wb{\,\linebreak[0]} \def\wB {$\,$\wb}
 \newcommand\Bi[2]    {\bibitem[#2]{#1}}
 \newcommand\Epub[2]  {{\em #2}, {\tt #1}}
 \newcommand\inBo[8]  {{\em #8}, in:\ {\em #1}, {#2}\ ({#3}, {#4} {#5}), p.\ {#6--#7} }
 \newcommand\inBO[9]  {{\em #9}, in:\ {\em #1}, {#2}\ ({#3}, {#4} {#5}), p.\ {#6--#7} {\tt [#8]}}
 \newcommand\J[7]     {{\em #7}, {#1} {#2} ({#3}) {#4--#5} {{\tt [#6]}}}
 \newcommand\JJ[7]    {{\em #7}, {#1}{#2} ({#3}) {#4--#5} {{\tt [#6]}}}
 \newcommand\JO[6]    {{\em #6}, {#1} {#2} ({#3}) {#4--#5} }
 \newcommand\JP[7]    {{\em #7}, {#1} ({#3}) {{\tt [#6]}}}
 \newcommand\BOOK[4]  {{\em #1\/} ({#2}, {#3} {#4})}
 \newcommand\PhD[2]   {{\em #2}, Ph.D.\ thesis #1}
 \newcommand\Prep[2]  {{\em #2}, preprint {\tt #1}}

\def\adma  {Adv.\wb Math.}
\def\atmp  {Adv.\wb Theor.\wb Math.\wb Phys.}   
\def\coma  {Con\-temp.\wb Math.}
\def\comp  {Com\-mun.\wb Math.\wb Phys.}
\def\cpma  {Com\-pos.\wb Math.}
\def\inma  {Invent.\wb math.}
\def\jktr  {J.\wB Knot\wB Theory\wB and\wB its\wB Ramif.}
\def\jmst  {J.\wb Math.\wb Sci.\wb Univ.\wB Tokyo}
\def\joal  {J.\wB Al\-ge\-bra}
\def\jomp  {J.\wb Math.\wb Phys.}
\def\jopa  {J.\wb Phys.\ A}
\def\jpaa  {J.\wB Pure\wB Appl.\wb Alg.}
\def\jram  {J.\wB rei\-ne\wB an\-gew.\wb Math.}
\def\maan  {Math.\wb Annal.}
\def\momj  {Mos\-cow\wB Math.\wb J.}
\def\nupb  {Nucl.\wb Phys.\ B}
\def\pnas  {Proc.\wb Natl.\wb Acad.\wb Sci.\wb USA}
\def\ruma  {Revista de la Uni\'on Matem\'atica Argentina}
\def\slnm  {Sprin\-ger\wB Lecture\wB Notes\wB in\wB Mathematics}
\def\taac  {Theo\-ry\wB and\wB Appl.\wB Cat.}
\def\tams  {Trans.\wb Amer.\wb Math.\wb Soc.}
\def\toap  {Topology\wB Applic.}
\def\topo  {Topology}
\def\trgr  {Trans\-form.\wB Groups}

  \end{document}